%% file: carlinifestasilwolf_preprint.tex
\RequirePackage{fix-cm}
\documentclass[smallextended]{svjour3}       
\smartqed  
\usepackage{graphicx}
\usepackage{amsmath}
\usepackage{amssymb}
\usepackage{color}
\usepackage{graphicx}
\usepackage{epsfig}
\usepackage{epstopdf}
\usepackage[applemac]{inputenc}
\usepackage{enumitem}

\input mymacro

\begin{document}

\title{A Semi-Lagrangian scheme for a modified version of the Hughes model for pedestrian flow
}

\titlerunning{A SL scheme for the Hughes model for pedestrian flow}        

\author{Elisabetta Carlini\and Adriano Festa \and Francisco J. Silva \and Marie-Therese Wolfram  }
 

\institute{E. Carlini \at
              Dipartimento di Matematica ``G. Castelnuovo'', Sapienza Universit\`a di Roma, \\ \email{carlini@mat.uniroma1.it}        
  \and
           A. Festa \at
              RICAM -- Johann Radon Institute
for Computational and Applied Mathematics, 
Austrian Academy of Sciences (\"OAW), 
              \email{adriano.festa@oaew.ac.at}
              \and 
              F.J. Silva \at
              XLIM - DMI 
UMR CNRS 7252
Facult\'e des Sciences et Techniques, Universit\'e de Limoges,  
\email{francisco.silva@unilim.fr}. 
\and
M-T. Wolfram \at
Mathematics Institute, University of Warwick, Coventry CV4 7AL and 
              RICAM -- Johann Radon Institute
for Computational and Applied Mathematics, 
Austrian Academy of Sciences (\"OAW), 
              \email{m.wolfram@warwick.ac.uk}
}

\date{Received: date / Accepted: date}

\maketitle

\begin{abstract}
In this paper we present a Semi-Lagrangian scheme for a regularized version of the Hughes model for pedestrian flow. Hughes originally proposed a coupled nonlinear PDE system describing the evolution
of a large pedestrian group trying to exit a domain as fast as possible. The original model corresponds to a system of a conservation law for the pedestrian density and an Eikonal equation to determine
the weighted distance to the exit. We consider this model in  presence of small diffusion and discuss the numerical analysis of the proposed
Semi-Lagrangian scheme. Furthermore we illustrate the effect of small diffusion on the exit time with various numerical experiments.
\end{abstract}

\section{Introduction}\label{intro}

\noindent   In the last decades crowd dynamics has attracted the attention of many researchers   in the scientific community. Starting from the field of applied physics and transportation research,  the motion
of pedestrian crowds raised more and more interest in the applied mathematics community. 

\noindent Mathematical models range from the microscopic level, where the individual dynamics 
are described separately, to the mesocopic and macroscopic level, where the distribution with respect to their velocity and/or position in space is considered.\\
Microscopic models are either force-based, such as the social force model proposed by Helbing and co-workers \cite{helbing1995social} or lattice based like the cellular automata models proposed in \cite{burstedde2001simulation,blue2001cellular}. On the macroscopic level the evolution of the pedestrian density is usually described by a conservation law, see for example  \cite{hughes2002continuum,colombo2005pedestrians,piccoli2011time,colombo2012,Degond}. In these models the velocity field may depend on the current local density, a given external potential and physical constraints due to walls and/or barriers. Recently mean field games, cf. \cite{huang2006large,LasryLions07}, have been proposed to model the evolution of large pedestrian crowds, see \cite{Lachapelle10}. These models can be derived from  stochastic optimal control problems for multi-agent systems as the number of individuals tends to infinity. For a detailed overview on different modeling approaches in pedestrian dynamics we refer to \cite{bellomo2011modeling,CPT_BOOK}.
  
\noindent In 2002 R. Hughes proposed a macropscopic model for pedestrian dynamics in \cite{hughes2002continuum}, which is based on a continuity equation (describing the evolution of the crowd density) and an Eikonal equation (giving the shortest weighted distance to an exit). It is given by
\be\label{hughes}
\left\{
\ba{ll}
\partial_t m(x,t)-\hbox{div}(m(x,t)\, f^2(m(x,t))\nabla u(x,t))=0,  \\[4pt]
|\nabla u(x,t)|=\displaystyle\frac{1}{f(m(x,t))},
\ea\right.\ee
where $x\in \Omega$ denotes the position in space, $t\in (0,T]$, $T\in \RR_+$ the time and   $\nabla$ the gradient with respect to the space variable $x$. The function $m$ corresponds to the pedestrian density and $u$ the weighted shortest distance to a target, for example an exit. Hughes proposed different functions penalizing regions of high density, the simplest choice being $f(m)=1-m$ where $1$ corresponds to the maximum scaled pedestrian density. In this work, we will assume that $f$ is a general smooth function. \\
System \eqref{hughes} is a highly nonlinear coupled system of partial differential equations. Few analytic results are available, all of them restricted to spatial dimension one. The main difficulty comes from the low regularity of the potential $u(x,t)$, which is only Lipchitz-continuous. For existence and uniqueness results of {a} regularized problem in 1D and the corresponding Riemann problem we refer to \cite{di2011hughes,amadori2012one,amadori2014existence}.

\noindent In this work we consider a modified version of \eqref{hughes}, which served as the basis for the $1$D analysis presented by Di Francesco et al. in \cite{di2011hughes}. It corresponds to
\be\label{principaleq}
\left\{
\ba{ll}
\partial_t  m(x,t) -  \eps \Delta m(x,t) - \mbox{div}(m(x,t)\, f^2(m(x,t))\nabla u(x,t) ) = 0, \\ 
-\eps \Delta u(x,t) + \half | \nabla u(x,t)|^{2} = \displaystyle\frac{1}{2f^2(m(x,t))+ \delta}.  
\ea \right.
\ee
in $\Omega \times (0,T)$.\\
The regularization parameter $\delta > 0$ prevents the blow-up of the cost when approaching the maximum density one. The  diffusive terms allow to use standard  analytical techniques from nonlinear PDE theory, see \cite{di2011hughes}. Diffusive phenomena have been observed and studied in pedestrian dynamics \cite{thompson1995computer,liu2015modeling}, giving an additional justification of the modification considered.\\
System \eqref{principaleq} has to be supplemented with suitable boundary and initial conditions. We consider an initial density $m_0$ of the agents satisfying that $m_0\geq 0$, $m_0\in L^{\infty}(\Omega)$ and the support of $m_0$ is a subset of $\Omega$.  Note that rescaling the density $m_0$, and possibly modifying the function $f$ in the equation,  we can assume that $\int_{\Omega}m_0(x)\dd x=1$. This normalization is useful in order to provide a probabilistic interpretation of the Fokker-Planck (FP) equation in   \eqref{principaleq}. Possible boundary conditions for the pedestrian density $m$ at the exit are:
\begin{itemize}
\item a given fixed outflow, corresponding to Neumann boundary condition,
\item an outflux which depends on the pedestrian density, hence a Robin boundary condition,
\item or a prescribed pedestrian density, giving a Dirichlet boundary condition.
\end{itemize}
Let $\T$ denote the common \emph{target/goal} of the crowd, which is a subset of the boundary i.e. $\T\subset\partial \Om$.
We set the pedestrian density to $m = 0$ at the target, hence individuals immediately leave the domain. On the rest of the boundary we impose homogeneous Neumann boundary conditions, i.e. individuals can not penetrate the walls. For the Eikonal equation we set $u=0$ at the target and a suitable Dirichlet boundary condition on the rest of the boundary.  The above conditions can be summarized as follows:
\be\label{BC}\left\{
\ba{ll}										 
m(x,0)= m_0(t),  &\text{ on  } \Omega\times\{0\},\\[4pt]
m(x,t) = 0, \; \;& \text{ on } \T \times (0,T),\\[4pt]
u(x,t) = 0, \; & \text{ on } \T \times (0,T),\\[4pt]
u(x,t)= g(x)  \; \; &\text{ on } \;  { \partial  \Omega \setminus\T} \times (0,T),  \\[4pt]
(\eps\nabla m + f^2(m)\nabla u \,m)(x,t)\cdot \hat{n} (x)= 0, \; \; &\text{ on } \partial \Omega\setminus\T \times (0,T),
\ea \right.
\ee
where $\hat n$ denotes the outer normal vector to the boundary, which is assumed to be { smooth}. {  Since the theoretical analysis of \eqref{principaleq}-\eqref{BC} has been done in \cite{di2011hughes} in 1D with homogeneous Dirichlet  boundary conditions, rather than tackling the theoretical analysis of  \eqref{principaleq}-\eqref{BC},  in this work we focus on the efficient numerical discretization  as we detail below.} 

\noindent Semi-Lagrangian (SL) schemes have been  successfully used to discretize Hamil\-ton-Jacobi-Bellman (HJB) equations, see \cite{falconeferrettilibro} and the references therein.
They are based on approximating the  characteristics of the problem. A SL scheme has been presented in \cite{CS15} to deal with linear  FP equations. In this work, we use a SL scheme to  numerically solve the stationary HJB equation in \eqref{principaleq}.  We propose an extension of the scheme in \cite{CS15} in order to deal with nonlinear FP equations posed on a bounded domain.   

 \noindent One of the main advantages of SL schemes is that they are explicit and   allow  large time steps. 
This is of special relevance since we are interested in the   behavior of the solutions for arbitrary values of the horizon $T$ which can be large (for example, if we are interested in the evacuation time). Moreover the SL discretization allows us to run stable simulations for small regularization parameters, closer in the spirit  to the original hyperbolic system proposed by Hughes.

\noindent This paper is structured as follows: in Section \ref{prelim} we introduce the necessary preliminaries, including the trajectiorial interpretation of both equations,  to  present and study the  SL discretizations in Section \ref{numerics}. In Section \ref{simulations} we illustrate the  influence of the diffusivity on different performance parameters, such  as the evacuation time of the crowd or the formation of congestions.

\section{Preliminaries}\label{prelim}

\noindent In this section we recall the stochastic optimal control interpretation of the HJB  { as well as the probabilistic interpretation of solutions of  FP  equations} and introduce some notations
used throughout this paper. 

\noindent Let $\Omega \subset \RR^d$ denote a bounded domain with { a smooth} boundary $\partial \Omega$. Assume that the common target of the crowd is on part of the boundary $\partial \Omega$, hence $ \T \subset \partial \Omega $.

 Let us consider a probability space $(\Omega, \F, \mathbb{F}, \mathbb{P})$ (where $\F$ is a $\sigma$-algebra, $\PP$ is a probability measure on $\F$, $\mathbb{F}:= (\F_{s})_{s \geq 0}$ is a filtration in $(\Omega, \F)$, i.e. $\F_{s} \subseteq \F $ for all $s \geq 0$ and $\F_{s_{1}} \subseteq  \F_{s_2}$ for all  $0\leq s_1\leq s_2$). We assume that $\mathbb{F}$ satisfies the usual hypothesis (see e.g. \cite{Protter-book}). We denote by $\EE$ the expectation operator in this probability space.

%

\paragraph{Trajectorial interpretation of the HJB equation.} It is well known  that the classical solution $u$ of the first  equation of \eqref{principaleq} can be represented as the value function of an associated stochastic optimal control problem, { which we recall now}. Given   a process $\alpha$ adapted to $\mathbb{F}$ (i.e. $\alpha(s)$ is $\F_s$-measurable for all $s$) and satisfying that $\EE\left(\int_{0}^{s}|\alpha(r)|^2 \dd r\right) < \infty$ for all $s\geq 0$ (we say that $\alpha$ is admissible), and  $x\in \ov{\Omega}$, we define
\be\label{fictivedynamics}
\ba{c}
y_{x,\alpha}(s) = x{+}\int_{0}^{s} \alpha(r) \dd r + \sqrt{2{\eps}}  W (s)   \hspace{0.5cm} \mbox{for all } \; s> 0, \\[6pt]
 \mbox{and } \; \;   \tau_{x,\alpha}:= \inf\{s > 0 \; ; \; y_{x,\alpha}(s) \in \partial \Om\}, 
\ea \ee  
where $W$ is a $d$-dimensional Brownian motion { adapted to $\mathbb{F}$}. Note that the time $\tau_{x,\alpha}$, which corresponds to the first time {the trajectory $y_{x,\alpha}$} leaves the domain $\Omega$, is a stopping time for the filtration $\mathbb{F}$ (i.e. $\{\tau_{x,\alpha}\leq s\} \in \F_s$ for all $s$). Let us fix $t\in [0,T]$. Classical results in stochastic control theory {\rm(}see e.g. {\rm \cite{FleSon92})} imply that, if $m(\cdot,t)$ is regular enough, then
\begin{align}\label{representationformulastochasticcontrol} 
\begin{split}
u(x,t)= \inf_{\alpha} \; \left\{\EE\left(\int_{0}^{\tau_{x,\alpha}}\left[ \half |\alpha(s)|^2 + (2{f^2(m(y_{x,\alpha}(s),t)})+ \delta)^{-1}\right]\dd s \right. \right.{}\\
+  g(y_{x,\alpha}(\tau_{x,\alpha})) \Big)\bigg\},
\end{split}
\end{align} 
and the optimal feedback law is given by {$ \alpha^{\ast} (x,t)= -\nabla u(x,t)$}   for all $s\geq 0$. The function $g$ is supposed to be strictly positive and taking sufficiently large values on $\partial \Om \setminus \T$ to {incite} that agents move towards the target $\T$. 

\noindent The dependence on the time variable $t$, {seen as a parameter in \eqref{representationformulastochasticcontrol}}, merits some additional comments. Indeed, the dependence of $u$ on $t$ is due exclusively to the  local density $m(x,t)$  on the right-hand-side of the HJB equation. This implies that  the trajectories $y_{x,\alpha} (\cdot)$ in  \eqref{fictivedynamics}  are {\it fictive}  {in the sense that  in the optimization process agents take into account the current pedestrian distribution $m(\cdot,t)$ only}.  This a fundamental difference to 
mean field game models (see \cite{LasryLions07,Lachapelle10,MR3199781}) and mean field type control problems (see \cite{MR3134900,MR3395471}), in which individuals anticipate the future dynamics of the crowd.

\paragraph{Trajectorial interpretation of the nonlinear FP equation.}  The trajectorial interpretation of the nonlinear FP equation is provided through stochastic differential equations of McKean-Vlasov type (or mean field type), see \cite{MR0221595,MR0233437,MR1431299,MR1108185}.   More precisely, let us consider the Stochastic Differential Equation (SDE)
\begin{equation}\label{Mackeanvlasov}
\ba{rcl} \dd X(t)&=& b(X(t), \mu(X(t),t),t)\,\dd t+ \sqrt{2\eps}\, \dd W(t), \hspace{0.4cm} \mbox{for all } \; t\geq 0, \\[6pt]
			X(0)&=& X^0,

\ea
\end{equation}
where $b: \RR^d \times \RR \times \RR_+ \to \RR^d$ is a regular vector-valued function,  $X^0$ is a random vector in $\RR^d$, independent of the Brownian motion $W(\cdot)$, with density $m_0$, and  $\mu(\cdot,t)$ is the density of $X(t)$. 
It can be shown  (see \cite{MR1653393}) that \eqref{Mackeanvlasov} admits a unique solution and that  $\mu$ is the unique classical solution of the nonlinear FP equation
\be\label{nonlinearFP}\ba{rcl}
\partial \mu -\eps\Delta \mu+ \mbox{div}(b(x,\mu,t) \mu) =0 \hspace{0.3cm} &\mbox{in }& \; \RR^d \times [0,\infty[, \\[6pt]
\mu(\cdot, 0)= m_0(\cdot) \hspace{0.3cm} &\mbox{in}& \;\RR^d.\ea
\ee
Therefore, if we set
\be\label{driftb} b(x,m,t):=-\nabla u(x,t)f^2(m(x,t))\ee
and  working on $\RR^d$ instead of $\Om$,  
equation \eqref{Mackeanvlasov} provides a formal probabilistic interpretation of the second   equation in \eqref{principaleq} with $m(\cdot,t)$ being the density of $X(t)$. Let us point out that the interpretation is a priori only  heuristic since $u$ depends  implicitly on $m$. Therefore the definition of $b$ in \eqref{driftb} does not actually fit the framework of \cite{MR1653393}, where the dependence on the density is explicit. \\
\noindent The probabilistic interpretation sketched above  is the basis of our SL scheme to solve \eqref{principaleq},  presented in the next section.   To include boundary conditions in the FP equation in  \eqref{principaleq} we reflect the discrete trajectories at $\partial \Om \setminus \T$ and truncate them at $\T$, see {\cite{Bossy2004,Gobet2001}}. 
 \\
\noindent
Finally, note that  in contrast to Mean Field Games, the model considered in this work does not impose dual boundary conditions for the HJB and the FP equation.

%

\section{The numerical scheme}\label{numerics}
In this section we propose a SL scheme to approximate the solution of \eqref{principaleq}. The crucial point is the discretization of the nonlinear FP equation, which is based on the fact that its solution is a measurable selection of the time-marginal densities of the diffusion defined by \eqref{Mackeanvlasov}  (see \cite{MR1653393}). We will first propose a SL scheme for a general nonlinear FP equation with smooth coefficients and a given velocity field depending explicitely on the density of the underlying stochastic process. We will prove that our scheme is consistent in an appropriate sense. The main feature of the scheme, which can be seen as an extension to the nonlinear case of the scheme proposed in \cite{CS15,CS13},  is that it is explicit and, at the same time, allows  large time steps. This is not the case for e.g. explicit finite-difference schemes where the consistency property is achieved under {the classical parabolic} CFL condition.

\noindent In the case of system \eqref{principaleq} the velocity field in the nonlinear FP equation depends implicitly on the density $m$ through the solution $u$ of the HJB. Therefore, in order to find an approximation of the velocity field we must solve the stationary  HJB equation at each time step. This is done in Section  \ref{hjbdiscretization}, where an   adaptation of the fully-discrete scheme proposed in \cite{CamFal95}, taking into account the Dirichlet boundary condition is presented.   Finally, in Section \ref{semilagrangiancomplete} we merge both schemes to provide the fully-discrete scheme for \eqref{principaleq}.

\noindent  Let us begin by introducing some standard notation. For simplicity, we  suppose that $\Omega=(0,L)^d$. Even if  this set $\Omega$ (and also the domains considered in the numerical simulations) {has not a smooth boundary}, we prefer to work on a square domain  in order to simplify the scheme.  Given a time step $\Delta t>0$ and a space discretization parameter $\Delta x>0$, let $M\in \NN$ and $N\in \NN$ be such that $M \Delta x=L$ and $N\Delta t=T$. Let us set $(x_{i}, t_{k}):= (i \Delta x , k \Delta t)$, where $i\in \{0,\hdots,M\}^{d}$ and  $k=0, \hdots, N$. For a given $A\subseteq \Omega$ we set $\mathcal{G}_{\Delta x} (A):=\{i \in \{0,\hdots, M\}^{d} \; : \; x_i\in A \}$ and call   $B(\mathcal{G}_{\Delta x}(A))$ and $B(\mathcal{G}_{\Delta x,\Delta t}(A))$   the spaces of  grid functions defined on $\{x_i : i\in \mathcal{G}_{\Delta x}(A) \}$ and $\{(x_i,t_k), i\in \mathcal{G}_{\Delta x} (A), k=0,\hdots, N\}$ respectively.

\noindent  Given a standard uniform triangulation  of $\ov \Omega$ with   vertices  belonging to $\mathcal{G}_{\Delta x}(\ov \Omega)$, we denote by $\{{\beta_{i}} \; ; \; i \in \mathcal{G}_{\Delta x}(\ov \Omega)\}$ the set of $\mathbb{P}_1$-basis functions associated to this  triangulation. We recall that  ${\beta_i}$  are continuous functions, affine on each simplex and $\beta_i(x_j)=\delta_{i j}$ for all $j \in \mathcal{G}_{\Delta x} (\ov{\Omega}) $   (where $\delta_{i,j}$ denotes the Kronecker symbol). Moreover, the functions $\beta_i$ have compact support and satisfy that  $0\leq \beta_i \leq 1$ and $\sum_{i\in  \mathcal{G}_{\Delta x} (\ov{ \Omega})}\beta_i(x)=1$ for all $x\in \ov\Omega$. We consider the following  linear interpolation operator on $\ov{\Omega}$ 
\be\label{definterpolation}
I[u](\cdot):=\sum_{i\in  \mathcal{G}_{\Delta x} (\ov{\Omega})}u(x_i)\beta_i(\cdot) \hspace{0.2cm} \mbox{for } u \in B(\mathcal{G}_{\Delta x}(\ov{\Omega})).
\ee


\subsection{A Semi-Lagrangian scheme for a nonlinear Fokker-Planck equation }
In this section we propose a SL scheme to numerically solve the following nonlinear FP equation
 \be\label{FPb}\left\{
\ba{ll}
 \partial_t  m -  \eps \Delta m + \mbox{div}(m\, b(x,m,t) ) = 0&\hspace{0.3cm} \mbox{in $\RR^d \times (0,T)$}, \\[6pt]							 
 m(\cdot,0)= m_0(\cdot)  &\hspace{0.3cm} \mbox{in $\RR^d$},
\ea \right.
\ee
where  
$b:\Omega\times\RR \times[0,T]\to\RR^d$ is a given smooth vector field, depending on $m$. By an abuse of notation we denote by $m_0$ the smooth initial datum, now defined on $\RR^d$ with compact support.
 
\noindent In order to formally derive the scheme, we multiply the first equation in \eqref{FPb} by a smooth test function  ${\phi}$ with compact support 
and integrate by parts to  get:
\begin{align}\label{weaksol}
&\int_{\RR^d} \phi(x) m(x,t_{k+1})\dd x=\int_{\RR^d}\phi(x) m(x,t_{k})\dd x \\
&+\int_{t_k}^{t_{k+1}} \int_{\RR^d}
 [ b(x,m(x,t),t)\cdot  \nabla \phi(x)+\eps \Delta \phi(x) ] m(x,t)\dd x\dd t.\nonumber
 \end{align}
We first approximate \eqref{weaksol} as
 \begin{align*}
\int_{\RR^d} \phi(x) &m(x,t_{k+1})\dd x=\\
&\int_{\RR^d}[\phi(x)+ \Delta t
b(x,m(x,t_k),t_k) \cdot \nabla \phi(x)+{\Delta t}\eps \Delta \phi(x) ] m(x,t_k)\dd x . 
 \nonumber
 \end{align*}
 Using a Taylor expansion we obtain
 \begin{align*}
\int_{\RR^d} \phi(x) &m(x,t_{k+1})\dd x=\\
&\frac{1}{2d}\sum_{\ell=1}^{d} \int_{\RR^{d}}[\phi(x+ \Delta t
 b(x,m(x,t_k),t_k)+ \sqrt{{2d\eps}\Delta t} {\bf{e}}_\ell )] m(x,t_k)\dd x +\nonumber\\
&\frac{1}{2d}\sum_{\ell=1}^{d} \int_{\RR^d}[\phi(x+\Delta t
 b(x,m(x,t_k),t_k)- \sqrt{{2d\eps}\Delta t }{\bf{e}}_\ell) ] m(x,t_k)\dd x,\nonumber
 \nonumber
 \end{align*}
 where ${\bf e}_{\ell}$ denotes the $\ell$-th canonical vector in $\RR^d$.\\
We define
\begin{align}\label{e:approxm} 
\ba{c}E_i=[x_i^1- \half \Delta x, x_i^1+ \half \Delta x]\times \hdots \times [x_i^d- \half \Delta x, x_i^d+ \half \Delta x], \\[6pt]
m_{i,k}:= {\frac{1}{(\Delta x)^{d}}}\int_{E_i} m(x,t_k) \dd x.\ea
\end{align}
 {Approximating the integrals of the form  $ \int_{E_j}c(x)m(x,t_{k'})\dd x$ by sums  $(\Delta x)^d c(x_j) m_{j,k'}$, where $c$ is a smooth function, $j\in \ZZ^d$ and $k'=0, \hdots, N$, we get}
 \begin{align}\label{sum}
\sum_{{j\in \ZZ^d}} &\phi(x_j) m_{j,k+1}=\\
&\frac{1}{2d} \sum_{\ell=1}^{d}\sum_{j\in \ZZ^d} \phi( \Phi^{\ell,+}_{j,k}[m(x_j,t_k)])  m_{j,k}+
\frac{1}{2d}  \sum_{\ell=1}^{d}\sum_{j\in \ZZ^d}  \phi( \Phi^{\ell,-}_{j,k}[m(x_j,t_k)] ) m_{j,k},
 \nonumber
 \end{align}
 where, for $\mu \in \RR$, $j\in \ZZ^d$, $k=0,\hdots, N-1$ and $\ell=1,\hdots, d$, { we have defined}
\be\label{car} 
\Phi^{\ell,\pm}_{j,k}[\mu]:= x_{j}+\Delta t \,b(x_j,\mu,t_k)\pm  \sqrt{2d\eps \Delta t}  {\bf e}_{\ell}.
\ee
Given $i\in \ZZ^d$ setting $\phi=\beta_i$ in { \eqref{sum}, we have}
\be\label{sum2}
m_{i,k+1}= \frac{1}{2d}  \sum_{j\in \ZZ^d}\sum_{\ell=1}^{d}
\left(\beta_{i} (\Phi ^{\ell,+}_{j,k}[m(x_{j},t_{k})])
       +\beta_{i}(\Phi^{\ell,-}_{j,k}[m(x_{j},t_{k})]) \right) m_{j,k}.
\ee
Finally, since $ m_{i,k} \simeq m(x_i,t_k)$,  setting $m_k=(m_{i,k})_{i\in \ZZ^d}$, { \eqref{sum2}} gives the following explicit scheme for $m_{i,k}$:  
\be\label{schemefp}
\ba{rl}
m_{i,k+1}&= G(m_k,i,k) \hspace{0.4cm} \forall \; k=0,\hdots, N-1, \; \; i\in \ZZ^{d}, \\[8pt]
m_{i,0}&=\frac{ \int_{E_{i}}m_{0}(x) \dd x}{(\Delta x)^d}  \hspace{0.4cm} \forall i\in \ZZ^{d},
\ea 
\ee
in which the nonlinear operator $G$ is defined by
\be\label{definicionG}
 G (w,i,k) := \frac{1}{2d}  \sum_{j\in \ZZ^d}\sum_{\ell=1}^{d}
\left(\beta_{i} \left(\Phi ^{\ell,+}_{j,k}\left[ w_j \right]\right)
       +\beta_{i}\left(\Phi^{\ell,-}_{j,k}\left[ w_j\right]\right) \right) w_j,
\ee 
  for every $w \in B(\ZZ^d)$. Because of the explicit in time discretization the scheme is well-defined. Given the solution $m_{i,k}$ of \eqref{schemefp}, we associate the function $m_{\Delta x, \Delta t}: \RR^{d}\times [0,T]   \to \RR$ defined as:
\be\label{definicionmdensity}
m_{\Delta x, \Delta t}(x,t):= m_{i,k} \hspace{0.3cm} \mbox{if $x\in E_i$ and $t\in [t_{k}, t_{k+1}[$,  $i\in \ZZ^d$, $k=0,\hdots,N$}.
\ee
Note that the scheme is conservative by definition, i.e.
\begin{align*}
\int_{\RR^{d}} m_{\Delta x, \Delta t}(x,t_k) \dd x= (\Delta x)^d\sum_{i\in \ZZ^d}m_{i,k} = \int_{\RR^{d}} m_{0}(x) \dd x \quad  \text{ for all } k=1,\hdots, N. 
\end{align*}
We extend \eqref{definicionG} to $B(\ZZ^d)\times \RR^d\times [0,T]$ by { defining} 
$$
\ba{rl}
G_{\Delta x,\Delta t}(v,x,t):=& G(v,i,k) \hspace{0.3cm}\mbox{if $x\in E_i$} \\[6pt]
\;  & \mbox{and $t\in [t_{k}, t_{k+1}[$,  $i\in \ZZ^d$, $k=0,\hdots,N-1$.}
\ea
$$ 
Following similar computations as in the derivation of the scheme, we can prove that \eqref{schemefp} is consistent. {The consistency result in the following Proposition is called} {\it weak}  {in order} to
underline  consistency to the weak formulation {of \eqref{FPb}}. 

\begin{proposition}[Weak consistency]\label{consistenciadebil} Assume that $m: \RR^d \times [0,T] \to \RR_+$ satisfies:
\begin{itemize}
\item[\textbullet] $\int_{\RR^d}m(x,t)\dd x$ is uniformly bounded in $[0,T]$. \vspace{0.1cm}
\item[\textbullet]     For all $t \in [0,T]$, $m(\cdot,t)\in C^{2}(\RR^{d})$ and for all $x\in \RR^d$, $m(x,\cdot)$ is Lipschitz with a Lipschitz constant independent of $x$.
\end{itemize} 
Set  $ m_{i,k}$ and $m_{\Delta x, \Delta t}$ as in  \eqref{e:approxm} and \eqref{definicionmdensity}.
Then, assuming that   $b$ is Lipschitz, for every 
$\phi \in  C_{0}^{\infty}\left(\RR^d\right)$ and $k=0, \hdots, N$ we obtain 
\be\label{consistencyseparada}
\int_{\RR^{d}} \phi(x) m_{\Delta x, \Delta t}(x,t_k)\dd x = \int_{\RR^{d}}\phi(x) m(x,t_k) \dd x + O(\Delta x), 
\ee 
and  for $k=0, \hdots, N-1$ 
\begin{align}\label{consistencyseparada2}
\begin{split}
\int_{\RR^{d}}& \phi(x) G_{\Delta x,\Delta t}(m_k, x, t_k) \dd x\\
&= \int_{\RR^{d}} \phi(x) m(x,t_k)\dd x  + \int_{t_k}^{t_{k+1}} \int_{\RR^{d}} b(x,m(t,x),t) \cdot \nabla \phi(x) m(x,t) \dd x \dd t \\
&+\int_{t_k}^{t_{k+1}} \int_{\RR^{d}} \eps \Delta \phi(x)m(x,t)\dd x \dd t + O(\Delta x +(\Delta t)^2).
\end{split}
\end{align}
In particular, if $m$ is differentiable w.r.t. to the time variable and if $(\Delta x_n, \Delta t_n)$ is a sequence of space and time steps such that
\begin{align*}
(\Delta x_n, \Delta t_n)\to 0 \text{ and }\Delta x_n / \Delta t_n \to 0
\end{align*} as $n\to \infty$, then   
\begin{align}\label{limitedebil}
\begin{split}
&\lim_{n\to \infty}\frac{1}{\Delta t_n} \int_{\RR^{d}} \phi(x)\left[ m_{\Delta x_n, \Delta t_n}(x,t_{k^n+1})- G_{\Delta x_n,\Delta t_n}(m_{k^n}, x, t_{k^n}) \right]\dd x \\
&~~= \int_{\RR^{d}} \phi(x) \left[ \partial_{t}m(x,t)-\eps \Delta m(x,t)+ \mbox{{\rm div}} \left(b(x,m(x,t),t) m(x,t) \right)\right] \dd x,
\end{split}
\end{align}
for $k^n$ such that $t_{k^n} \to t$.
\end{proposition} \smallskip

\begin{proof} Let $C=\mbox{supp}(\phi)$, which is a compact set. By definition 
\begin{align*}
\int_{\RR^{d}} \phi(x) m_{\Delta x, \Delta t}(x,t_k)\dd x&= \sum_{i \in  \G_{\Delta x}(C)}  m_{i,k} \int_{E_i} \phi(x) \dd x \\[6pt]
\; &=  \sum_{i \in \G_{\Delta x}(C)}  \int_{E_{i}} m(x_i,t_k)\phi(x) \dd x  + O((\Delta x)^2)\\[6pt]
\; &= \sum_{i \in \G_{\Delta x}(C)} \int_{E_{i}} m(x,t_k)\phi(x) \dd x  + O(\Delta x+(\Delta x)^2)\\[6pt]
\; &= \int_{\RR^{d}} m(x,t_k) \phi(x) \dd x +O(\Delta x)
\end{align*}
where, we have used that
\begin{equation*}
 m_{i,k} = \frac{1}{(\Delta x)^d}\int_{E_i}m(x,t_{k})\dd x=m(x_i,t_k)+O(\Delta x^2),
\end{equation*}
which holds true by a Taylor expansion, since $m$ is regular. On the other hand, 
\begin{align*}
\int_{\RR^{d}} \phi(x) G_{\Delta x,\Delta t}(m_k, x, t_k) \dd x&= \sum_{ i \in \G_{\Delta x}(C)} G(m_k,i,k) \int_{E_i} \phi(x) \dd x. 
\end{align*}
Now,
\begin{multline*}
  G(m_k,i,k)\int_{E_i}\phi(x) \dd x = \sum_{j\in \ZZ^d}\frac{1}{2d} \sum_{l=1}^{d} \sum_{s \in \{+,-\}} \beta_{i}(\Phi_{j,k}^{\ell,s}[m_{j,k}] )m_{j,k}\int_{E_i}\phi(x) \dd x,\\[6pt]
 =\sum_{j\in \ZZ^d}\frac{1}{2d} \sum_{l=1}^{d} \sum_{s \in \{+,-\}} \beta_{i}(\Phi_{j,k}^{\ell,s}[m_{j,k}]) \int_{E_j}m(x,t_k)\dd x  \phi(x_i)  + O(\Delta x),
\end{multline*}
  
\noindent where we have used that 
$$\frac{1}{(\Delta x)^d} \int_{E_{i}} \phi(x) \dd x= \phi(x_i) + O(\Delta x).$$
Therefore, since 
$$ \sum_{i \in \G_{\Delta x}(C)} \beta_{i}(\Phi_{j,k}^{\ell,s}[m_{j,k}]) \phi(x_i)= I[\phi]( \Phi_{j,k}^{\ell,s}[m_{j,k}])= \phi( \Phi_{j,k}^{\ell,s}[m_{j,k}])+O((\Delta x)^2),$$
interchanging the sums w.r.t. $i$ and $j$, we get 
\begin{multline}\label{antesdetaylor}
\int_{\RR^{d}} \phi(x)\; G_{\Delta x,\Delta t}(m_k, x, t_k) \dd x\\
= \sum_{j\in \ZZ^d}\int_{E_j}m(x,t_k)\dd x \frac{1}{2d} \sum_{l=1}^{d} \sum_{s \in \{+,-\}}\phi( \Phi_{j,k}^{\ell,s}[m_{j,k}]) 
 +O((\Delta x)^2+\Delta x).
\end{multline}
Note that for $x\in E_j$
\begin{align*}
\frac{1}{2d} \sum_{l=1}^{d}& \sum_{s \in \{+,-\}}\phi( \Phi_{j,k}^{\ell,s}[m_{j,k}])\\
&= \phi(x_j)+\Delta tb(x_j,m_{j,k},t_k) \cdot \nabla \phi(x_j) + \eps \Delta t \Delta \phi(x_j) + O((\Delta t)^2), \\
&= \phi(x)+ \int_{t_k}^{t_{k+1}}\left[b(x,m(x,t),t)\cdot \nabla \phi(x) + \eps  \Delta \phi(x)\right]\dd t \\
& \phantom{= \phi(x)} + O(\Delta x+ \Delta x \Delta t+ (\Delta t)^2).
\end{align*}
The  equality in \eqref{consistencyseparada2} follows easily from the relation above and \eqref{antesdetaylor}. Finally, relation \eqref{limitedebil} follows directly from \eqref{consistencyseparada} and an integration by parts in the space variable. 
\hfill \small $\blacksquare$ \normalsize
\end{proof}

\noindent We conclude by discussing the implementation of Dirichlet and Neumann boundary conditions. Hence we consider the nonlinear FP on the bounded domain with mixed Dirichlet and Neumann boundary conditions
 \be\label{FPwithboundaryconditions}\left\{
\ba{ll}
 \partial_t  m -  \eps \Delta m + \mbox{div}(m\, b(x,m,t) ) = 0&\hspace{0.3cm} \mbox{in $\Om \times (0,T)$}, \\[6pt]							 
 m(\cdot,0)= m_0(\cdot)  &\hspace{0.3cm} \mbox{in $\Om$},\\[6pt]
 m= 0&\hspace{0.3cm}\mbox{in $\T$},\\[6pt]
 (\eps\nabla m - b(x,m,t) m)\cdot \hat{n} =  0, \; \; &\hspace{0.3cm} \mbox{on $\partial \Omega\setminus\T \times (0,T)$} \\[6pt]
\ea \right.
\ee

\noindent Note that $\Phi$ defined in \eqref{car} (with $\mu=m(x_j,t_k)$) can be interpreted as a single Euler step in time of
\be\label{SP}\ba{rcl}
\dd X(s)&=&b\left(X(s),m(X(s),s),s\right)\dd s+\sqrt{2\eps}\dd W(s), \hspace{0.5cm} s \in (t_k,t_{k+1}),\\[6pt]
X(t_k)&=&x_j,\ea
\ee
with a random walk discretization of the Brownian motion $W(\cdot)$.
Indeed, considering a random vector $Z$ in $\RR^d$ such that for all $\ell=1,\dots,d$
\begin{align}
\label{defZ}
\begin{split}
&\PP( Z^{\ell}=1)=  \PP( Z^{\ell}=-1)= \frac{1}{2d} \\
&\PP\Big( \bigcup_{1\leq \ell_1 < \ell_2 \leq d } \{  Z^{\ell_1}\neq 0\} \cap \{  Z^{\ell_2}\neq 0\}  \Big)=0,
\end{split}
\end{align}
the function $\Phi_{j,k}^{\ell,\pm}[m(x_j,t_k)]$ corresponds to one realization of 
$$x_j+\Delta t b(x_j,m(x_j,t_k),t_k)+ \sqrt{2d \eps \Delta t }Z.$$ 
{Taking into account the boundary conditions in \eqref{FPwithboundaryconditions},  the discretization of} the stochastic process \eqref{SP} driving the evolution of the density $m(\cdot,t)$ has to be reflected in $\partial \Omega\setminus\T \times (0,T)$ and truncated at $\T$, as
for example in \cite{Gobet2001,Bossy2004} in the case of a consistent Monte-Carlo simulation. 

\noindent Let us introduce the notation  $s\wedge t:=\min\{s,t\}$ for $t$ and $s$ belong  to $\RR$, with the convention that  $+\infty \wedge s=s$ for all $s\in [0,\infty)$.
In order to proceed
 for $i\in \G_{\Delta x}(\Omega)$ and  $\mu\in \RR$, denoting for all $\ell=1,\dots,d$
\begin{align*} 
&\widehat{\Delta t}_{i}^{\ell, \pm }=  \inf\{ \gamma>0 \; ;  \; x+\gamma  b(x_i, \mu,t_k)\pm  \sqrt{2  d\eps  \gamma}{\bf e}_{\ell} \in\T\} \wedge \Delta t, 
\end{align*}
we redefine \eqref{car} as
\begin{equation*} \label{Projcar1} 
 \Phi_{i,k}^{\ell,\pm} [ \mu ] := x_{i}+\widehat{\Delta t}_{i}^{\ell,\pm}b(x_i,\mu,t_k)\pm\sqrt{2d\eps \widehat{\Delta t}_{i}^{\ell,\pm}}  {\bf {e}}_{\ell}.
\end{equation*}
We approximate the Neumann boundary condition in $\partial \Om \setminus \T$ by  the {\em symmetrized Euler scheme}  proposed in \cite{Bossy2004}.
 We define the  symmetrized characteristics  as $P( \Phi_{i,k}^{\ell,\pm} [ \mu ])$, where 
 $P:\RR^d\to \Omega$  is defined as (see   Fig. \ref{proje})
\be\label{Projection} 
 P (z):=\begin{cases}
z, &{\rm{if}}\;z\in \ov{ \Omega},\\
2w^*-z, \qquad\, \mbox{where } \; w^*:=\underset{w\in\Omega}{\rm{argmin}} |z-w|,&{\rm{if}}\;z\notin \ov{ \Omega}.
\end{cases}
\ee
We get the following scheme to approximate \eqref{FPwithboundaryconditions}
\be\label{schemefpwithboundary}\ba{rcl}
m_{i,k+1}&=& G(m_k, i,k) \hspace{1cm}\; \mbox{if } \;i\in \mathcal{G}(\ov \Omega\setminus \T), \\[6pt]
m_{i,k+1}&=&0  \hspace{2.58cm} \mbox{if } i\in \G_{\Delta x}(\T),\\[6pt]
m_{i,0}&=& \frac{\int_{E_{i}} m_{0}(x) \dd x}{(\Delta x)^d} \hspace{1.1cm} \mbox{for all  }  i\in \G_{\Delta x}(\ov \Omega),
\ea
\ee
where $G$ is redefined accordingly as
$$G(m_k, i,k):= \frac{1}{2d}  \sum_{j\in \ZZ^d }\sum_{\ell=1}^{d}
\left[\beta_{i} \left( P( \Phi_{j,k}^{\ell,+} [m_{j,k}]  )\right)
       +\beta_{i}\left( P ( \Phi_{j,k}^{\ell,-} [m_{j,k}])\right)  \right] m_{j,k}.
$$
\begin{remark} Evidently, if we consider the scheme \eqref{schemefpwithboundary} the statements of Proposition  \ref{consistenciadebil} still hold true for test functions with compact support {\it strictly contained} in $\Omega$.
\end{remark}
\begin{figure}[t]
\begin{center}
\includegraphics[height=5cm]{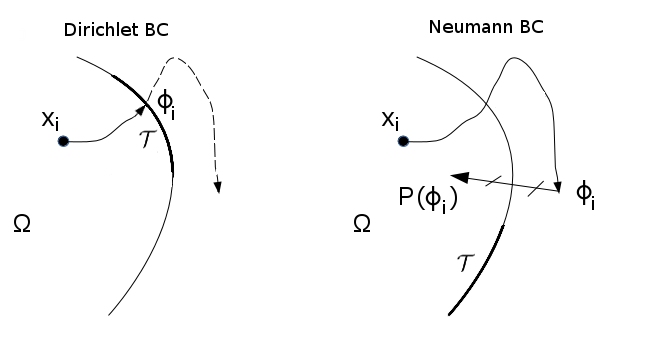}
\caption{Sketch of the implementation of the boundary conditions when the characteristics leave the domain $\Omega$. } \label{proje}
\end{center}
\end{figure}

\subsection{A Semi-Lagrangian scheme for the Hamilton-Jacobi-Bellman equation}\label{hjbdiscretization}
In this section we consider a SL scheme to solve the HJB equation and thus to approximate the velocity field $-\nabla u(x,t)$.  In order to formally introduce  the scheme, we recall that the solution $u(x,t)$ of the HJB equation in \eqref{principaleq} depends on $t$ only because of the presence of $m(x,t)$ on the r.h.s. of the equation, which is supposed to be given.  In the entire section  we fix $t>0$. Rewriting the quadratic term in its Legendre-Fenchel transfom, the solution $u(x,t)$ of the HJB equation is the  unique classical solution of (see \cite[Chapter 4]{FleSon92})
\begin{align}\label{hjbwithinf}
&\inf_{\alpha \in \RR^{d}}\left\{ \half |\alpha|^{2} -  \alpha\cdot \nabla u(x)  +  \eps \Delta u(x)\right\}+ \frac{1}{2f^2(m(x,t))+\delta}= 0 \; \: \; \; \mbox{if $x\in \Om$}, \\ &u (x)=g(x) \; \; \qquad\mbox{if $x\in \partial \Omega$}.\nonumber
\end{align}
Using that     the optimal $\alpha$ in the infimum of \eqref{hjbwithinf} is given by $\alpha= \nabla u (x,t)$, {setting} $\A:=\{ \alpha \in \RR^{d} \; ; \; |\alpha|\leq \|  \nabla u(\cdot,t)\|_{L^{\infty}}\}$ and $F(x,t):= 1/(2f^2(m(x,t))+\delta)$,  {system} \eqref{hjbwithinf} can be rewritten as 
\begin{align}\label{hjbwithinfmuwithA}
&\min_{\alpha \in \A}\left\{ \half |\alpha|^{2} -  \alpha\cdot \nabla u(x)  + \eps \Delta u(x)\right\}+ F(x,t)= 0, \; \; \; \mbox{if $x\in \Omega$},\\
& u(x)=g(x) \qquad \mbox{if $x\in \partial \Omega$},\nonumber
\end{align}
where, for {notation} convenience, we have suppressed the dependence of $u$ on $t$, since the time variable is fixed. \\
In order to numerically solve \eqref{hjbwithinfmuwithA}, let us introduce a {\it fictive} time step $h>0$.  Recalling the definition of $y_{x,\alpha}$ (trajectory) and $\tau_{x,\alpha}$ (first time that the trajectory leaves the domain) in \eqref{fictivedynamics}  and noting that  $\tau_{x,\alpha}\wedge h$ is a stopping time, then for any admissible $\alpha$ and $x\in \Omega$,    It\^o's formula yields
\begin{multline*} \EE\left(u(y_{x,\alpha}(\tau_{x,\alpha} \wedge h))-u(x) \right)=\\ \EE\left( \int_{0}^{\tau_{x,\alpha} \wedge h} \alpha(s) \cdot \nabla u(y_{x,\alpha}(s)) \dd s + \eps \int_{0}^{\tau_{x,\alpha} \wedge h} \Delta u(y_{x,\alpha}(s)) \dd s\right).
\end{multline*} 
Formally, if we discretize the r.h.s. of the above expression as 
$$\EE\left(\tau_{x,\alpha} \wedge h\right) \left( \alpha \cdot  \nabla u(x)+  \eps\Delta u(x)\right)$$ (where $\alpha \in \A$ is arbitrary) and setting $\tilde{h}:= \tau_{x,\alpha} \wedge h$, we get the following approximation of \eqref{hjbwithinfmuwithA} 
\begin{align}\label{approximationHJB}
&\tilde{u} (x)= \min_{\alpha\in \A} \EE\left(  \tilde{u}(y_{x,\alpha}(\tilde{h})) +\frac{\tilde{h}}{2}|\alpha|^{2} +\tilde{h} F (x,t)\right) \; \; \; \mbox{if $x\in \Omega$}, \hspace{0.3cm} \\
&\tilde{u} (x)=g(x) \; \; \hspace{5cm}\qquad \mbox{if $x\in \partial \Omega$}.\nonumber
\end{align}
Then, it is natural to approximate $\tilde{h}$ as 
$$\hat{h}:=\inf\{ \gamma>0 \; ;  \; x +\gamma \alpha+  \sqrt{2  d \eps  \gamma }Z \in \partial \Omega\} \wedge h,$$
where $Z$ is a random vector defined as in \eqref{defZ}. Denoting for all $\ell=1,\hdots,d$ 
$$
\ba{c}
\hat{h}_{x,\alpha}^{\ell,+}:=  \inf\{ \gamma>0 \; ;  \; x+\gamma \alpha+  \sqrt{2  d\eps \gamma}{\bf e}_{\ell} \in \partial \Omega\} \wedge h, \\[6pt]
 \hat{h}_{x,\alpha}^{\ell,-}:=  \inf\{ \gamma>0 \; ;  \; x+\gamma \alpha-  \sqrt{2  d\eps  \gamma}{\bf e}_{\ell} \in \partial \Omega\} \wedge h,
\\[6pt]
y_{x,\alpha}^{\ell,+}:= x+\hat{h}_{x,\alpha}^{\ell,+}\alpha +   \sqrt{2d\eps \hat{h}_{x,\alpha}^{\ell,+} }{\bf e}_{\ell},\; \; \;  \hspace{0.4cm}y_{x,\alpha}^{\ell,-}:= x+\hat{h}_{x,\alpha}^{\ell,-}\alpha -   \sqrt{2d \eps \hat{h}_{x,\alpha}^{\ell,-} }{\bf e}_{\ell}, \ea$$ \normalsize
and setting $\hat{h}_{x,\alpha}^{\ell,\pm}=\hat{h}_{x,\alpha}^{\ell,+}+\hat{h}_{x,\alpha}^{\ell,-}$, we get the following approximation of \eqref{approximationHJB} 
\be\ba{rcl}
\hat{u} (x)&=& \min\limits_{\alpha\in \A} \left\{\frac{1}{2d} \sum_{\ell=1}^{d}
\left[ \hat{u} (y_{x,\alpha}^{\ell,+})+ \hat{u} (y_{x,\alpha}^{\ell,-})+\frac{\hat{h}_{x,\alpha}^{\ell,\pm}}{2}|\alpha|^{2} +\hat{h}_{x,\alpha}^{\ell,\pm} F (x,t)\right] \right\},\\[6pt]  \label{approximationHJB1} 
\; & \; &  \hspace{6.2cm}\mbox{for  $x\in \Omega$}, \\[6pt]
\hat{u} (x)&=&g(x) \; \; \hspace{5.32cm}\mbox{for $x\in \partial \Omega$}.\ea
\ee 
Finally, in order to obtain the space discretization from \eqref{approximationHJB1}, given a space step $\Delta x >0$ we  interpolate $\hat u$ in space using the operator $I$ defined in \eqref{definterpolation}. Given $i \in \mathcal{G}_{\Delta x}(\ov \Omega)$ let us  set $y_{i,\alpha}^{\ell,+}:=y_{x_{i},\alpha}^{\ell,+}$ and  $y_{i,\alpha}^{\ell,-}:=y_{x_{i},\alpha}^{\ell,-}$ with analogous definitions for $\hat{h}_{i,\alpha}^{\ell,+}$, $\hat{h}_{i,\alpha}^{\ell,-}$  and $\hat{h}_{i,\alpha}^{\ell,\pm}$. For $v\in \B(\mathcal{G}_{\Delta x}(\ov \Omega))$ define  
$${W}(v,i):=\min_{\alpha\in \A} \left\{ \frac{1}{2d} \sum_{\ell=1}^{d} \left[I [v](y_{i,\alpha}^{\ell,+})+  I [v](y_{i,\alpha}^{\ell,-}) +    \frac{\hat{h}_{i,\alpha}^{\ell,\pm}}{2}|\alpha|^{2} +\hat{h}_{i,\alpha}^{\ell,\pm} F(x_i,t) \right]\right\}.$$ 
Thus, interpolating the unknown in the formula for $\hat{u}(x)$, we get the following fully-discrete scheme to approximate the solution $u $  of \eqref{hjbwithinfmuwithA}: \smallskip\\ 
Find $u  \in  B(\mathcal{G}_{\Delta x}(\ov{\Om})) $ such that 
\begin{equation}\label{SL}
\ba{rcl}
u_i&=&{W}(u,i) \hspace{0.3cm} \mbox{for all } i \in \mathcal{G}_{\Delta x}(\Omega),\; \\[6pt]
   u_{i}&=&g(x_{i}) \hspace{0.5cm}  \mbox{for all } i \in   \mathcal{G}_{\Delta x}(\partial \Omega).
\ea
\end{equation}
Note that, alternatively,  problem \eqref{SL} can be written in the form: \smallskip\\ 
Find $u \in  B(\mathcal{G}_{\Delta x}(\ov{\Om})) $ such that 
\begin{equation}\label{Shorward}
0= \max_{\alpha \in \A} \left\{(B^{\alpha}u)_{i}- c(\alpha)_{i}\right\} \hspace{0.5cm} \forall \; i \in \G_{\Delta x}(\ov{\Om}),   
\end{equation}
where {the linear operator} $B^\alpha:B(\mathcal{G}_{\Delta x}(\ov{\Om}))\to B(\mathcal{G}_{\Delta x}(\ov{\Om}))$  {and} $c(\alpha)$ are defined as following: 
\begin{align*}
 (B^{\alpha}v)_{i}:=& v_{i} - \frac{1}{2d}\sum_{j \in \G_{\Delta x}(\ov{\Om}), \; \ell=1, \dots, d}\left[\beta_{j}(y_{i,\alpha}^{\ell,+})+\beta_{j}(y_{i,\alpha}^{\ell,-})\right]v_j, \\
 c(\alpha)_i:=&   \frac{1}{2d}\sum_{\ell=1}^{d}\left[\half \hat{h}_{x_i,\alpha}^{\ell,\pm}|\alpha|^{2}+\hat{h}_{x_i,\alpha}^{\ell,\pm}F(x_i,t)\right],
\end{align*}
for every $i \in \G_{\Delta x}(\Om)$ {and  $v\in B(\mathcal{G}_{\Delta x}(\ov{\Om}))$}, and 
$$ (B^{\alpha}v)_{i}:= v_{i}, \hspace{0.5cm} c(\alpha)_i=g(x_{i}) \hspace{0.4cm} \mbox{for all } \; i\in \G_{\Delta x}(\partial{\Om}).
$$
In the following result we prove existence and uniqueness of a solution of \eqref{Shorward} using the \emph{policy iteration method}, which is also an efficient method to compute the solution (see Section \ref{simulations}).
%
%
%
\begin{lemma}\label{proofhoward} Problem \eqref{SL} admits a unique solution $u$. In addition, for   $\alpha^0$ arbitrary in $\A$, the sequence  defined by
\begin{align}\label{sequencehorward} 
& v^n:= (B^{\alpha^{n-1}})^{-1}c(\alpha^{n-1}), \\
& \alpha^n \in \mbox{{\rm argmax}}_{\alpha\in \A}\left\{B^{\alpha}v^{n}-c(\alpha)\right\}, \; \; \; \; n\geq 1,\nonumber
\end{align}
is well-defined. Furthermore, for all $i \in \G_{\Delta x}(\ov{\Om})$, the sequence $v^n_{i}$ is   non-increasing, converges to $u_{i}$, and  {any} limit point $\ov \alpha$ of $\alpha^n$ (there exists at least one) satisfies 
\begin{multline}\label{optimalcontrolconvergece}
0= (B^{\ov \alpha }u )_{i}- c(\ov \alpha )_{i}
=\max_{\alpha \in \A} \left\{(B^{\alpha}u )_{i}- c(\alpha )_{i}\right\} \hspace{0.5cm} \forall \; i \in \G_{\Delta x}(\ov{\Om}).
\end{multline}
\end{lemma}
\begin{proof}
The key point is to show that $B^\alpha$ is a monotone matrix for every $\alpha\in \A$. Indeed, we will show that $B^\alpha$ is of {\it positive type} (see \cite[Definition 6.4]{axelssonbook}), which, applied to our problem, means that in the directed graph associated to $B^\alpha$ from every node associated to an index in $\G_{\Delta x}(\Omega)$ there {exits} a path to a node associated to an index in $\G_{\Delta x}(\partial \Omega)$. It is well known that this property implies the monotonicity of $B^\alpha$ (see \cite[Theorem 6.5]{axelssonbook}).

\noindent Note that if we set $({\bf 1})_i:=1$ for all $i\in \G_{\Delta x}(\ov{\Om})$, we have that $(B^{\alpha}{\bf 1})_{i}=0$ if $i\in \G_{\Delta x}(\Om)$ and $(B^{\alpha}{\bf 1})_{i}=1$ if $i\in \G_{\Delta x}(\partial \Om)$. 
Since $\A$ is compact, if $h$ is small enough, the dominant terms in the definitions of $y_{i,\alpha}^{\ell,+}$ and $y_{i,\alpha}^{\ell,-}$ are $ \sqrt{2d\eps{h}}$ and $- \sqrt{2d\eps{h}}$, respectively, independently of $\alpha$. Therefore, starting from any point $i \in \G_{\Delta x}(\Om)$, there exists a sequence of indexes $i_0,..., i_m \in \G_{\Delta x}(\ov{\Om})$,  $\ell \in \{1,\hdots, d\}$ and $s\in \{+,-\}$   such that  $i_0=i$, $i_m \in \G_{\Delta x}(\partial \Omega)$ and  $\beta_{i_{j+1}}(y_{i_j,\alpha}^{\ell,s})>0$ for all $j=0,\hdots, m-1$ (indeed, we can choose the sequence $i_0, \hdots, i_m$ to be a subset of the set of indexes describing  the shortest path from $i$ to $\G_{\Delta x}(\partial \Om)$). Thus we have { shown} that $B^\alpha$ is of positive type and so it is monotone and, in particular, invertible. As a consequence of the invertibility,   the sequences $(\alpha^n, v^{n})$ in \eqref{sequencehorward} are well-defined. The proof  of the remaining assertions present no difficulties and are by now classical. We only sketch the main arguments and refer the reader  to  \cite{puterman1979convergence}, \cite{santos2004convergence} for detailed proofs. The existence of a solution of   \eqref{Shorward} can be deduced from the convergence of the sequence $v_i^{n}$ which follows from the fact that it is a non-increasing sequence (since $B^{\alpha}$ is monotone) and bounded   (which is a consequence of $\A$ being compact and $B^{\alpha}$ and $c(\alpha)$ being continuous w.r.t. $\alpha$). The uniqueness is an easy consequence of the fact that $B^{\alpha}$ is monotone while \eqref{optimalcontrolconvergece} follows directly using that $\A$ is compact (and so the sequence $\alpha^n$ has at least one converging subsequence) and the convergence of $v^{n}$. 
\hfill \small $\blacksquare$ \normalsize
\end{proof} 
\noindent 
\begin{remark}[Interpretation of the scheme as a Markov decision problem] Analogously to the stochastic control interpretation of the problem  in the continuous case sketched in Section \ref{prelim}, it is not difficult to see that  the scheme \eqref{SL} corresponds to the Bellman equation associated to the controlled Markov decision problem, for  $ {\hat{\imath}}$ in $\G_{\Delta x}(\ov{\Om})$
\be\label{generalizationshortestpath}
u_{\hat{\imath}}:=\inf  \left\{ \lim_{N\to \infty} \EE\left(\sum_{k=0}^{N-1}c(y_{\hat{\imath},k},\alpha_k(y_{\hat{\imath},k}))\right) \; ; \; 
 \; \mbox{over  $\alpha_k: \mathcal{G}_{\Delta x}(\ov{\Om}) \to \A$}\right\},\ee
 
 \noindent where $c: \G_{\Delta x}(\ov{\Om}) \times \A \to \RR$ is defined as 

$$c(j,\alpha):= \left\{\ba{ll}  \frac{1}{2d} \sum_{\ell=1}^{d}\hat{h}_{x_{j},\alpha}^{\ell,\pm}\left(\half|\alpha |^2+F(x_{j},t)\right)  & \mbox{if }  j \in \G_{\Delta x}(\Omega), \\[6pt]
								g(x_j)& \mbox{if }  j \in \G_{\Delta x}(\partial \Omega),\ea\right.$$								
and the Markov chain  $(y_{\hat{\imath},k})_{k\geq 0}$ in $\G_{\Delta x}(\ov{\Om})$, has   transition probabilities $\{p_{i,j} \; ; \; i, j \in \G_{\Delta x}(\ov{\Om})\}$, depending on the choice of the policies $(\alpha_{k})_{k\geq 0}$  and the initial distribution $p^0$,  given by
$$p_{i,j}= \frac{1}{2d} \sum_{\ell=1}^{d}\left[ \beta_{j}(y_{x_{i},\alpha}^{\ell,+})+\beta_{j}(y_{x_{i},\alpha}^{\ell,-})\right], \hspace{0.3cm} p_{j}^0= \delta_{\hat{\imath},j}$$
where $\delta_{i,j}=1$ if $i=j$ and $\delta_{i,j}=0$, otherwise. 
Note that problem \eqref{generalizationshortestpath} is a generalization of the well-known stochastic shortest path problem, for which the existence of a unique solution of the Bellman equation is well-known under a reachability condition over the target states {\rm(}see {\rm \cite[Chapter 7]{Ber95}} and Proposition 7.2.1 for a detailed proof of the aforementioned result{\rm)}.  This reachability condition corresponds exactly to a probabilistic reformulation of the property stating that $B^\alpha$ is of positive type (see the proof of Lemma \ref{proofhoward}).
\end{remark}
Finally, let us point out that following exactly the computations in  \cite[Proposition 1 (iii)]{CS15} (in a time dependent framework) we have the following {\it consistency} property of the scheme: let  $(\Delta x_{n}, h_{n})\to 0$ with $\Delta x_n^2=o({h_n})$ and consider a sequence of grid points $x_{i_{n}} \to x\in \Om$. Then, for every 
 $\phi \in  C^{\infty}\left(\Om\right)$, we have
 
\begin{equation*}\label{consistency} 
\lim\limits_{n\to \infty}  \frac{1}{h_n} \left[u_{i_n}- {W} (u,i_n)\right]   = -\eps\Delta \phi(x) +\half |\nabla \phi(x)|^{2}  -F(x,t). 
\end{equation*}

\subsection{A Semi-Lagrangian scheme for the system }\label{semilagrangiancomplete}
Now we have all the elements to present the numerical scheme for  the complete system \eqref{principaleq}-\eqref{BC}.
We use  scheme \eqref{schemefpwithboundary} to approximate the nonlinear  FP equation with 
drift $b(x,m,t)$ defined in \eqref{driftb}. In this case, the drift  depends also on the gradient  $\nabla u$ of the value function $u$, solution of the HJB equation,  which depends implicitly on $m(x,t)$.
We  proceed iteratively, in the following way: given the discrete measure $m_{k}$  at time $t_k$ ($k=0,\hdots, N-1$) , we  compute the discrete value function $u_k \in B(\mathcal{G}_{\Delta x}(\ov \Omega))$ by the scheme \eqref{SL} with
 $$
F(x_i,t_k):= 1/(2f^2(m_{i,k})+\delta), \qquad \delta>0.
$$ 
We 
denote by  $Du_{i,k}$  the discrete gradient, obtained by  centered finite differences of $u_{i,k}$ in in the {\it internal node} $x_i\in \Omega$ and by one side finite differences for the node on the boundary.
Then, we calculate $m_{k+1}$ with scheme \eqref{schemefpwithboundary} approximating the drift $\Phi^{\ell,\pm}_{j,k }[m_{j,k}]$ by 
$$
\Phi^{\ell,\pm}_{j,k}[m_{j,k}]\simeq x_{j}+\Delta t f^2\left( m_{j,k}\right)D u_{j,k}\pm  \sqrt{2d\eps \Delta t}  {\bf e}_{\ell}
$$
and we iterate the process until $k=N-1$.
\begin{remark} Note that the scheme is explicit in time, therefore the existence of a solution is a direct consequence of the construction, as opposite to analogous scheme for Mean Field Games (see \cite{CS12,CS13,CS15}) where the scheme is proved to be well-posed by a fixed-point argument. Again, this reflects at the discrete level the fact that in this model agents cannot anticipate the future behaviour of the crowd.
\end{remark} 

\section{Numerical simulations}\label{simulations}

In this section we illustrate the behavior of the solutions of system \eqref{principaleq} with various numerical experiments. The simulations are based on the SL
scheme presented in Section \ref{semilagrangiancomplete}. The most expensive task in the simulations corresponds to the computation of the solution of the HJB  at each time step.
This motivates a careful choice of the technique to compute the solution of \eqref{SL}. 
Despite the high efficiency of techniques  such as Fast Marching \cite{Set99,tsitsiklis1995} and Fast Sweeping methods \cite{Zhao05}, their applicability  is  essentially limited to first order equations. A  parallel algorithm for the numerical resolution of stationary  second order HJB equations has been proposed in \cite{CacFal15}. However the method is only efficient for sufficiently small diffusivities.
We follow a different but also well known strategy: the  \emph{policy iteration} method, as described in Lemma 1.
 This class of techniques is especially sensible to a good initial policy $ \alpha^0$ (see e.g. \cite{KAF13}).
In view of our scheme, presented in Section  \ref{semilagrangiancomplete}, a natural \emph{warm start} policy  at the time step $t_k$
 is the optimal policy $\ov \alpha$ obtained at the previous time step $t_{k-1}$, since we do not expect large variations of the density at each time step.\\
In order to compute the optimal policy  in \eqref{sequencehorward}, a common practice in the literature is to discretize the set $\A$ and select the minimizing policy  in this discrete set (see e.g. \cite{falcone06,KAF13}).
We proceed in this manner by choosing $n_\theta, n_\rho \in \NN$ and defining
\begin{equation*}
\alpha_{\rho,\theta}:=\rho(\cos\theta,\sin\theta), \qquad \theta\in\left\{i\frac{2\pi}{n_\theta},i=1,\dots, n_\theta \right\}, \; \rho\in\left\{ 0,1,\dots n_{\rho} \right\}.
\end{equation*}
In all the tests we set $n_\theta=32$, $n_{\rho}=4$. In the spirit of \cite{KAF13}, it could be also interesting to apply  more efficient minimization techniques to solve  \eqref{sequencehorward}.

\subsection{Exit scenario with two doors}
In our first example we study the exit behavior of a group of pedestrians from a room with two exits.
We are interested in how the regularization
parameter $\eps$ influences the splitting behavior of the crowd and  the evacuation time, i.e. the smallest time iteration $N_s$ such that $ m_{j,N_s} = 0$ for all {$j\in \G_{\Delta x}(\ov{\Om})$}. We set $\delta:=10^{-6}$ throughout this section.\\ 
The room is represented by the domain $\Omega:=(0,1)^2$; the initial mass has  a uniform density located at the center of the domain: 
\begin{equation*}
m_0(x):=\left\{
\begin{array}{ll}
M_0 \qquad & x\in [1/3,2/3]^2,\\
0 & \text{otherwise},
\end{array}\right.
\end{equation*}
where $M_0\in\RR_+$. 
 The discretization parameters $\Delta x$ and $\Delta t$, introduced in Section 3, and the parameter $h$, introduced in sub-Section 3.2, are set to $\Delta x=\Delta t=h=0.08$.
The uniform initial mass is  set to $M_0=0.7$ and the diffusion coefficient $\eps$ to $\eps=0.001$.
The set $\T$ corresponds to two exits of different width on opposite sides of the boundary: 
\begin{equation*}
\T:= \{0\}\times [0.13, 0.27]\cup \{1\}\times [0.49, 0.51].
\end{equation*}
The position of the two exits induces an asymmetric splitting of the crowd (see Figure  \ref{fig1} (left)).
Initially, a large part of the crowd chooses to move to the right exit. After a while this exit gets congested, inducing a part of the population to change objective. 
They decide to move towards the left exit instead of waiting at the right one, as shown in the right plot of Figure \ref{fig1}.  
We would like to mention that this 'turning-behavior' cannot be observed in a mean field game model, since in this case agents anticipate the future behavior of the crowd and would wait or take the
other exit right away (see \cite{CPT15}, Section 4.1 for some comparative tests). 

\begin{figure}[th]
\begin{center}
\includegraphics[height=4.3cm]{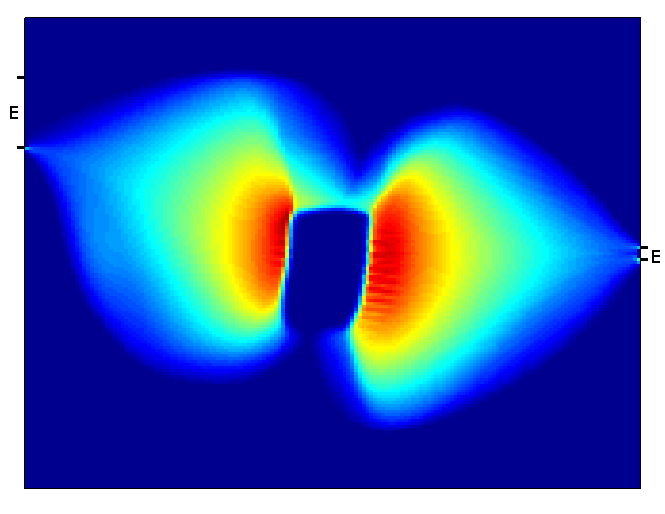} 
\hspace{0.5cm}
\includegraphics[height=4.3cm]{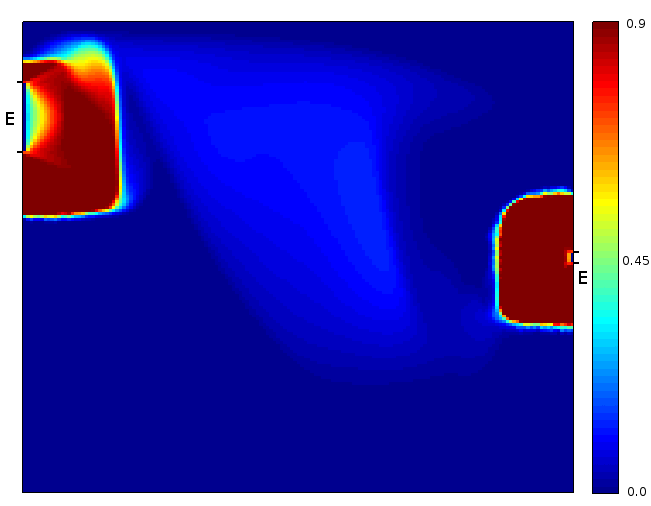}
\caption{Density contour lines at time $t=0.3$ (left) and $t=1.2$ (right). The exits are marked with the letter ``E''.} \label{fig1}
\end{center}
\end{figure}

\noindent In Table \ref{tab:test1}, we show the evacuation time for different values of $\eps$, as well as  the percentage of individuals exiting from the right or the left door. 
The diffusion influences the evacuation
time significantly: large values of $\eps$ prevent large pedestrian densities thus the effect of congestion is less evident; at the same time it increases the evacuation times. If  $\eps > 0.01$ the total mass of people that change their first objective is progressively reduced until it disappears (see Figure \ref{fig11}, where this effect is no longer observed). The reduction of the diffusive effects induces an initial decrease of the evacuation time that raises again because of congestion, for very small values of $\eps$.

\begin{table}[th]\label{tab:test1}
\begin{small}
\begin{center}
\begin{tabular}[b]{c|c|c|c}
        $\eps$    &  $N_{s}\Delta t$ & left exit & right exit \\
\hline\hline
{\boldmath{$4\cdot 10^{-2}$}}  & 5.08& 54.32 \% & 45.68 \%\\
\boldmath $2\cdot 10^{-2}$ & 4.62 & 53.72 \%& 46.27 \%\\
\boldmath $1\cdot 10^{-2}$  &  3.85 &53.40 \%& 46.59 \%\\
\boldmath $5\cdot 10^{-3}$  &  4.00 & 52.28 \% & 47.71 \%\\
\boldmath $2\cdot 10^{-3}$ & 4.10 & 52.17 \% & 47.82 \%\\
\boldmath $1\cdot 10^{-3}$ & 4.32 & 51.85 \% & 48.14 \%\\
\boldmath $5\cdot 10^{-4}$ & 4.77 & 51.40 \% & 48.59 \%\\
\end{tabular} \caption{Evacuation time and mass split for different values of $\eps$. }
\end{center}
\end{small}
\end{table}
\begin{figure}[th]
\centering
\includegraphics[width=0.45\textwidth]{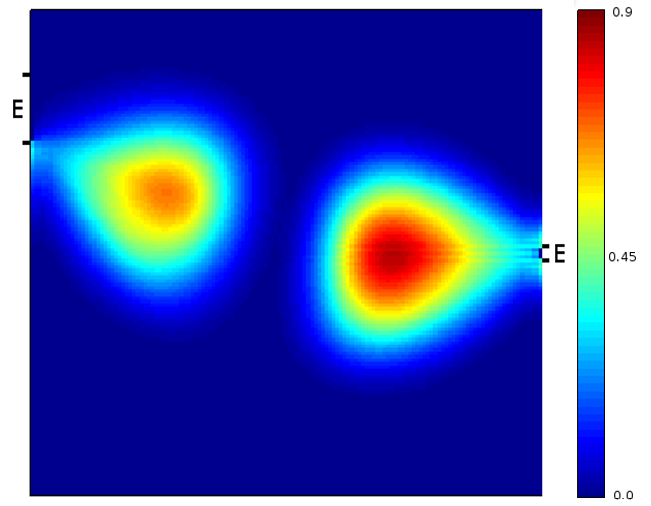}
\caption{Density contour lines at time $t=0.3s$ with $\eps=0.1$.} \label{fig11}
\end{figure}

\subsection{Exit scenario with barriers}

In this test we investigate the influence of barriers, usually called \emph{turnstiles}, on the evacuation time. Obviously the shape and the size of such barriers influence the dynamics of the system. We focus on the simple case where the barriers have a rectangular shape and fixed dimension. We vary the distance between them and consequently their number. 
We consider  the subset 
\begin{equation*}
\Gamma:=\left\{ x\in[0,1]^2\hbox{ s.t. } \min(0.1-|x-0.5|,0.02-|y-cs|)\geq 0,\, s\in \NN\cap[-4,4]\right\}
\end{equation*}
defined for a fixed $c\in\RR_+$. The choices $c=0.1,0.2,0.3$ correspond to $9, 7, 5$ barriers,  respectively resembling a fine/medium/coarse allocation. In this setup the computational domain corresponds to $\Omega:=[0,1]^2\setminus \Gamma$. 
We point out that in this case, because of the non convexity of the domain, the operator $P$ defined in \eqref{Projection} may be not well defined.  We overcome this by simply choosing the closer projection to the starting point of the discrete characteristic. \\
\noindent In the HJB equation we impose Dirichlet boundary conditions on $\partial \Gamma$   to preserve the regularity of the solution on the boundary. This point is related to reachability issues also discussed in \cite{aubin2009viability}. A practical approach to implement these conditions corresponds to introducing a narrow band of ``ghost nodes'' (see \cite[Section 5.1]{cristiani2007fast}) close to the barriers and setting a constant value $G$ such that $G> u_i$ for all $i\in \G_{\Delta x}(\ov \Omega)$ at these nodes. 

\noindent We start the simulation with an uniform distribution of individuals on the left side of the barriers which wants to exit through
$$
\T:= \{1\}\times [0.45, 0.55].
$$
The initial distribution corresponds to
$$
m_0(x):=\left\{
\begin{array}{ll}
M_0 \qquad & x\in [0.15, 0.35]\times [0.2,0.8],\\
0 & \text{ otherwise}.
\end{array}\right.
$$

\noindent We would like to understand how the barriers  and  the regularization parameter $\eps$ effect the evacuation time.
The evolution of the density is illustrated in Figure \ref{fig2tes}, in the case $M_0=0.7$,  $c=0.1$ and a diffusion coefficient $\eps=0.001$. 

\begin{figure}[t]
\begin{center}
\includegraphics[height=4.1cm]{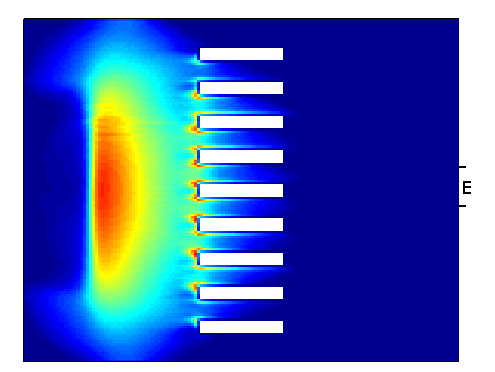} 
\hspace{0.3cm}
\includegraphics[height=4.1cm]{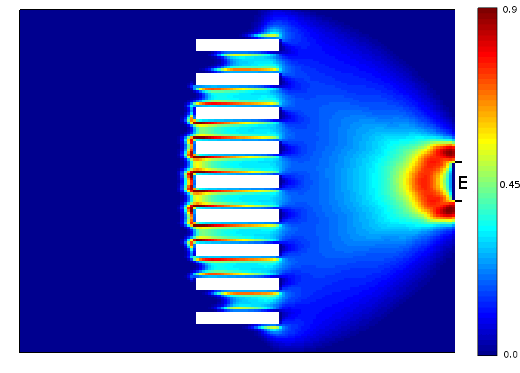}
\caption{Evolution of the density in test 2.  Contour lines at time $t=0.3$ and $t=1.2$. The exit is marked with an `E'.} \label{fig2tes}
\end{center}
\end{figure}

\noindent The number and the position of the turnstiles (determined by the parameter $c$) have a significant effect on the overall evacuation time.  In Table \ref{Tab1:test2} we observe the relation between exit time, fine/medium/coarse disposition of barriers  and diffusion parameter. The boxes highlight the shortest evacuation times for the different diffusivities. We observe that in the case of small diffusion a large number of turnstiles gives   the smallest evacuation time. This phenomenon disappears gradually for larger values of the diffusion parameter. 
The opposite behavior is observed when varying the total mass. In the case of a small total mass ($M_0=0.4,0.5$) high concentrations and congestion are less probable, and the configuration without barriers gives the lowest evacuation time. In the case of large initial masses, that is $M_0>0.5$, congested areas  start to appear. In this case the fine disposition of barriers improves the evacuation time (see Table \ref{Tab2:test2}).

\begin{table}[th]
\begin{small}
\centering
\begin{tabular}[b]{c|c|c|c|c}
 $\eps$   & $c=0.1$& $c=0.2$ & $c=0.3$  &  no barriers\\
\hline \hline
{\boldmath{$2\cdot 10^{-2}$} }&3.54   &  3.49  &  3.45  & \fbox{3.42}   \\
{\boldmath{$1\cdot 10^{-2}$}} &3.20 &   3.02   & \fbox{2.75} & 2.76     \\
{\boldmath{$5 \cdot 10^{-3}$}} &3.17   &  2.90 &  \fbox{2.70} & 2.85    \\
{\boldmath{$2\cdot 10^{-3}$}} &3.17    &   \fbox{3.02}   &  3.15  &3.32  \\
{\boldmath{$1\cdot 10^{-3}$}}  &3.50 &  \fbox{3.65}    &   3.70   & 3.75 \\
{\boldmath{$5\cdot 10^{-4}$}}  &\fbox{3.85}   &  4.12 & 4.3 & 4.25 \\
{\boldmath{$2\cdot 10^{-4}$}}  &\fbox{3.99} & 5.75& 5.85 & 5.85\\
{\boldmath{$1\cdot 10^{-4}$}} &\fbox{6.05} & 6.65&6.75 &  6.73\\
\end{tabular} 
\caption{Comparison of the evacuation time for different choices  of the coefficient $\eps$ and the parameter $c$.} \label{Tab1:test2}
\end{small}
\end{table}

\begin{table}[th]
\centering
\begin{small}
\begin{tabular}[b]{c|c|c|c|c}
        $M_0$   &  $c=0.1$& $c=0.2$ & $c=0.3$ & no barriers  \\
\hline \hline
\bfseries 1.0 & \fbox{5.05} & 5.25 &   5.30  &  5.40  \\
\bfseries 0.9 &  \fbox{4.55} & 4.85 & 5.12  & 5.55  \\
\bfseries 0.8  & \fbox{4.15} &  4.65 &  4.85  &  5.15  \\
\bfseries 0.7    &  \fbox{3.85} &  4.12   &  4.3   & 4.25 \\
\bfseries 0.6 & 3.55   & 3.45&    \fbox{3.42}   &   3.75    \\
\bfseries 0.5 &  3.35  &  3.22 &    3.25   &  \fbox{3.21}   \\
\bfseries 0.4 &   2.82  &  2.75  &      2.72     &    \fbox{2.45} \\
\end{tabular}
\caption{Comparison of the evacuation time for different choices of the parameter $c$ and varying $M_0$.} \label{Tab2:test2}
\end{small}
\end{table}

\subsection{The renovation of ``Les Halles'' in Paris}

In this last test, we consider a geometry representing a part of the  transport hub situated at   ``Les Halles'' in Paris. In this station we find connections  between underground lines, extra-city lines (RER) and buses. The center of the structure is the large transition zone shown in Figure \ref{fig3}. Recently the whole station was redesigned changing the shape of the hall. Furthermore, additional exits to improve the flow   of pedestrians have been added.

\begin{figure}[th]
\begin{center}
\includegraphics[height=4.2cm]{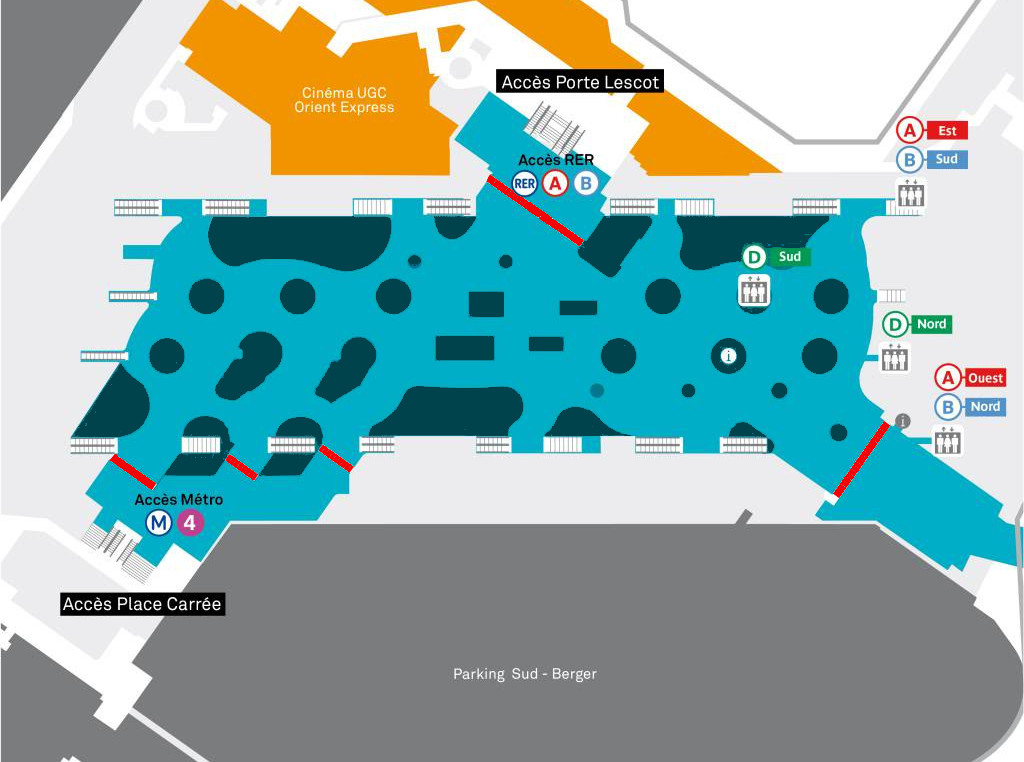} 
\hspace{0.3cm}
\includegraphics[height=4.2cm]{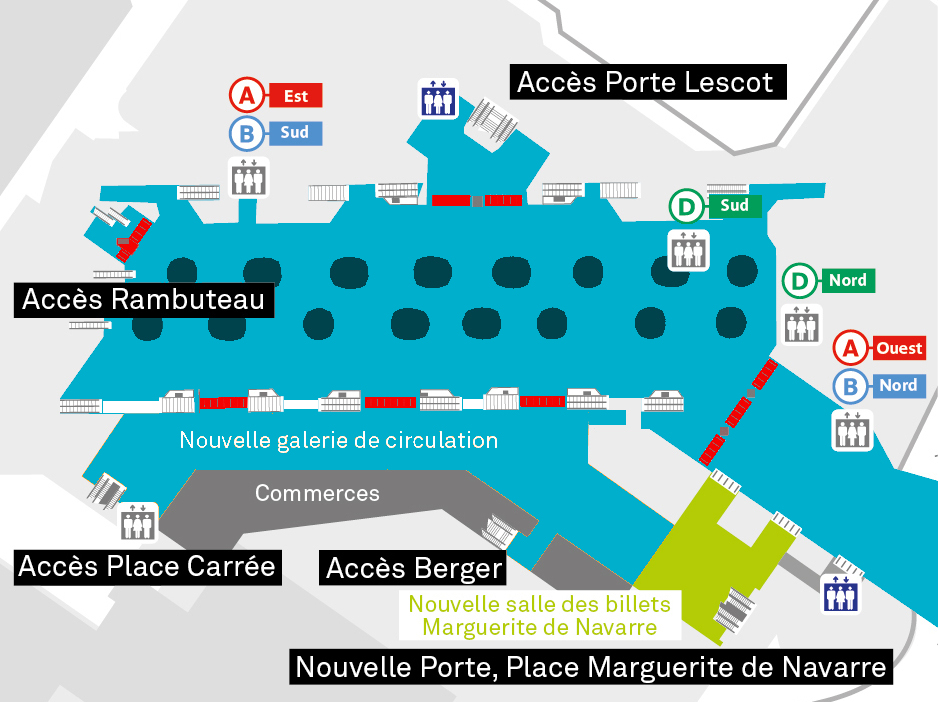}
\caption{Plan of the ``Salle d'exchange RER'' of Les Halles, Paris. In the 2014 (left) and in 2016 (right).} \label{fig3}
\end{center}
\end{figure}
\noindent In the following simulation, we want to model an evacuation scenario where a crowd of pedestrians placed in the center of the hall wants to reach the exits, located behind some turnstiles. 
 For simplicity we do not consider any additional   pedestrian inflows or the use of the elevators. 
We want to compare the evacuation capacity of the structure before and after the changes.\\

\noindent 
In our simulations, we set as computational domain $\Omega$    the light blue areas in Figure  \ref{fig3},   corresponding to the hall where  circulation is possible. The dark blue area corresponds to the obstacles. \\

\noindent In this test we set 
$$
f(x,m(x,t)):=\frac{1-m(x,t)}{\ell(x)},
$$
where $\ell(x)$ corresponds to the environmental running cost. It takes the value $2$ on the turnstiles  (in red in Figure  \ref{fig3}) and $1$ elsewhere. Even if in such case $f$ is discontinuous,  the simulations appear to be stable. We set also  $\eps:=0.001$.\\
\begin{figure}[t]
\begin{center}
\includegraphics[height=4.5cm,width=11cm]{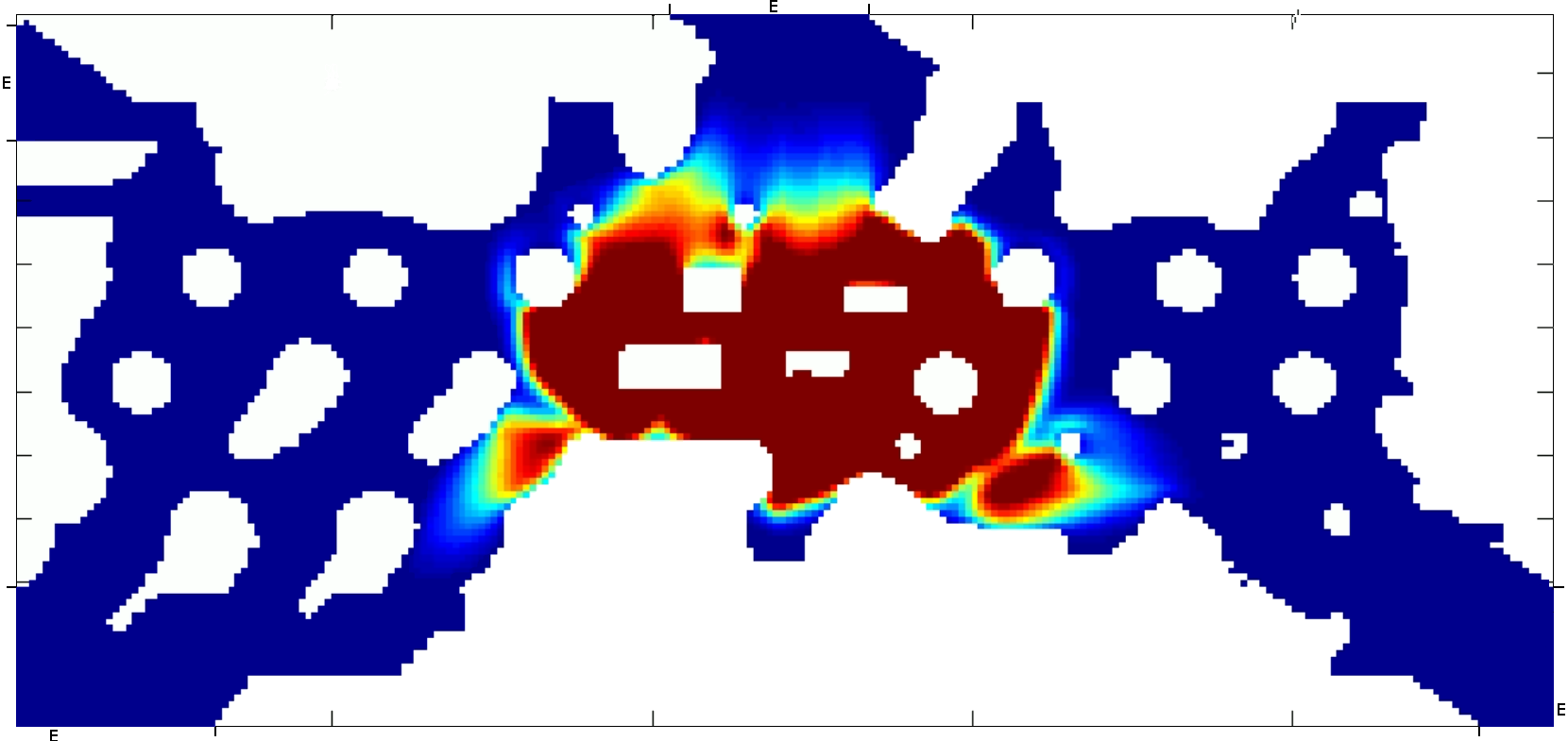}
\includegraphics[height=4.5cm,width=0.6cm]{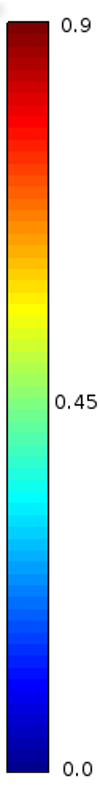}\\ 
\includegraphics[height=4.5cm,width=11cm]{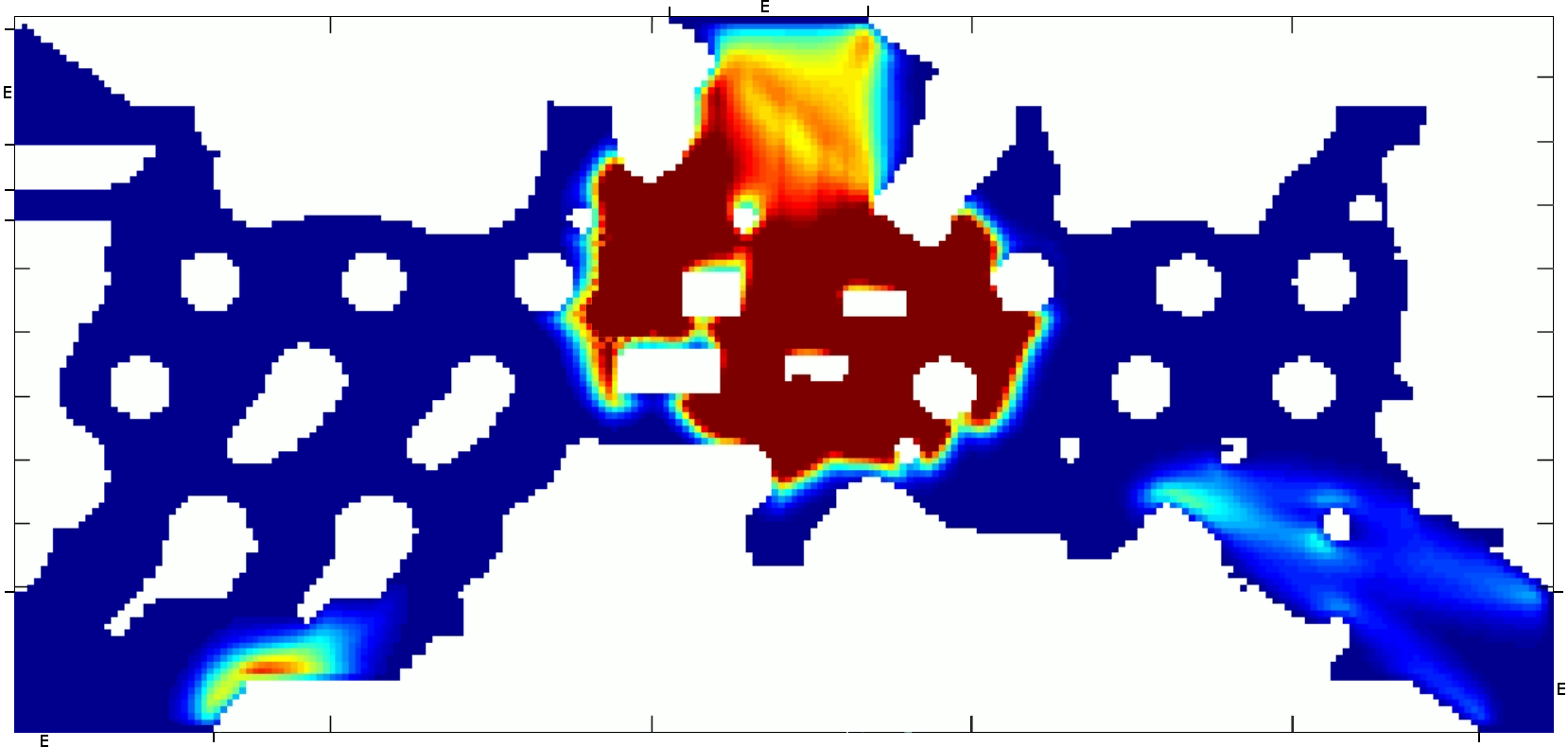}
\includegraphics[height=4.5cm,width=0.6cm]{bar.png}
\caption{Simulation for the structure in 2014. Contour lines at time $t=1s$ (top) $t=5s$ (bottom), with $M_0=0.7$. Exits are marked with an `E'.} \label{fig4}
\end{center}
\end{figure}  
The initial configuration of the crowd is chosen as
 \begin{equation*}
m_0(x):=\left\{
\begin{array}{ll}
M_0\qquad & x\in \left\{(\frac{d_1}{3},\frac{2 d_1}{3})\times (\frac{d_2}{3},\frac{2 d_2}{3})\right\}\cap \Omega,\\
0 & {\textrm{otherwise}}.
\end{array}\right.
\end{equation*}
where $[0,d_1]\times[0,d_2]$ is the smallest rectangle containing $\Om$.\\
In Figures \ref{fig4} and \ref{fig5}, we show the simulation of the evacuation  at time $t=1$ and $t=5 $ with $M_0=0.7$, respectively on the domain corresponding to the old structure and to the new one. 
The figures  illustrate the improvements in pedestrian circulation, mostly due to
a new exit added at the center of the lower boundary of the domain, which
 limits the  congestion created next to the main exit (located at the center of the top boundary). In Table \ref{fig66}, we compare the evacuation time for different initial masses $M_0$: as expected   the improvement of the evacuation capacities is more effective in the case of high pedestrian densities. 
\begin{figure}[th]
\begin{center}
\includegraphics[height=4.5cm,width=11cm]{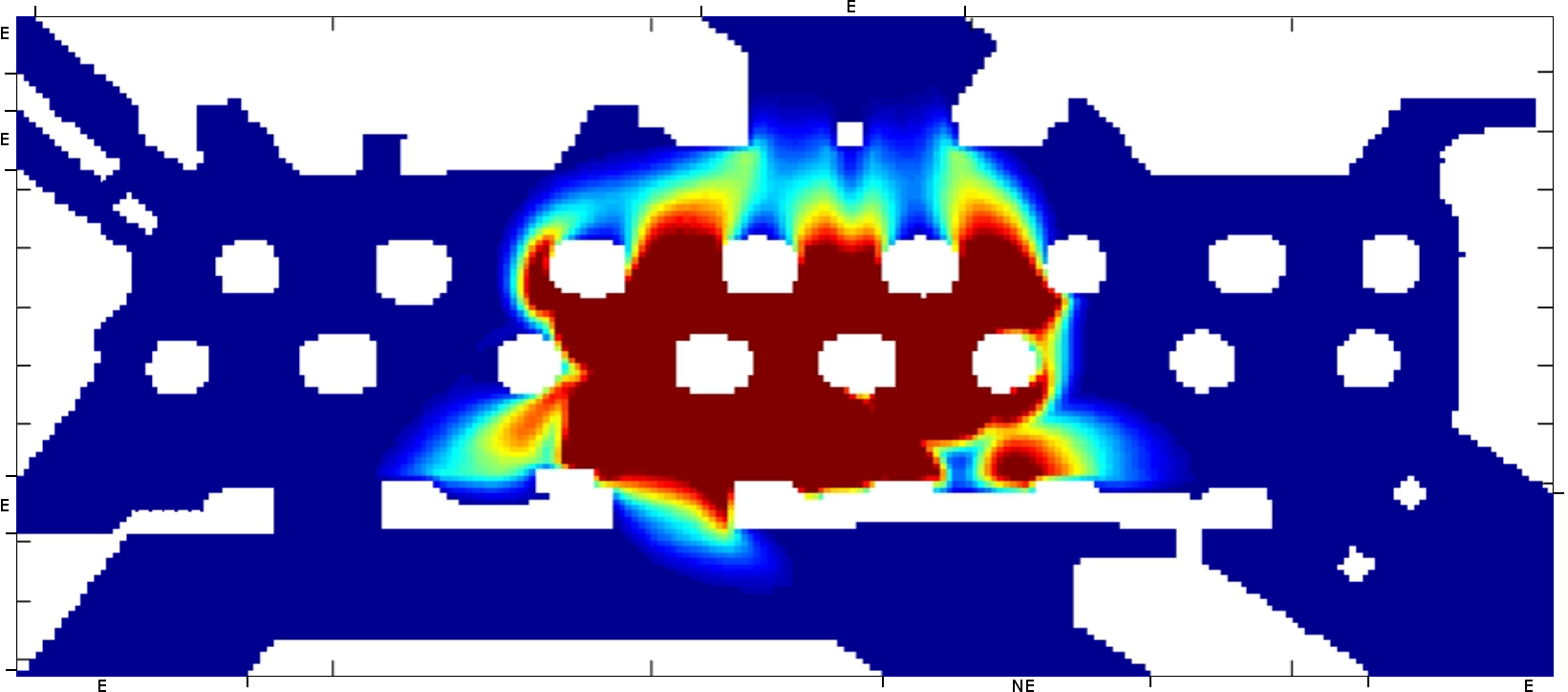}
\includegraphics[height=4.5cm,width=0.6cm]{bar.png}\\ 
\includegraphics[height=4.5cm,width=11cm]{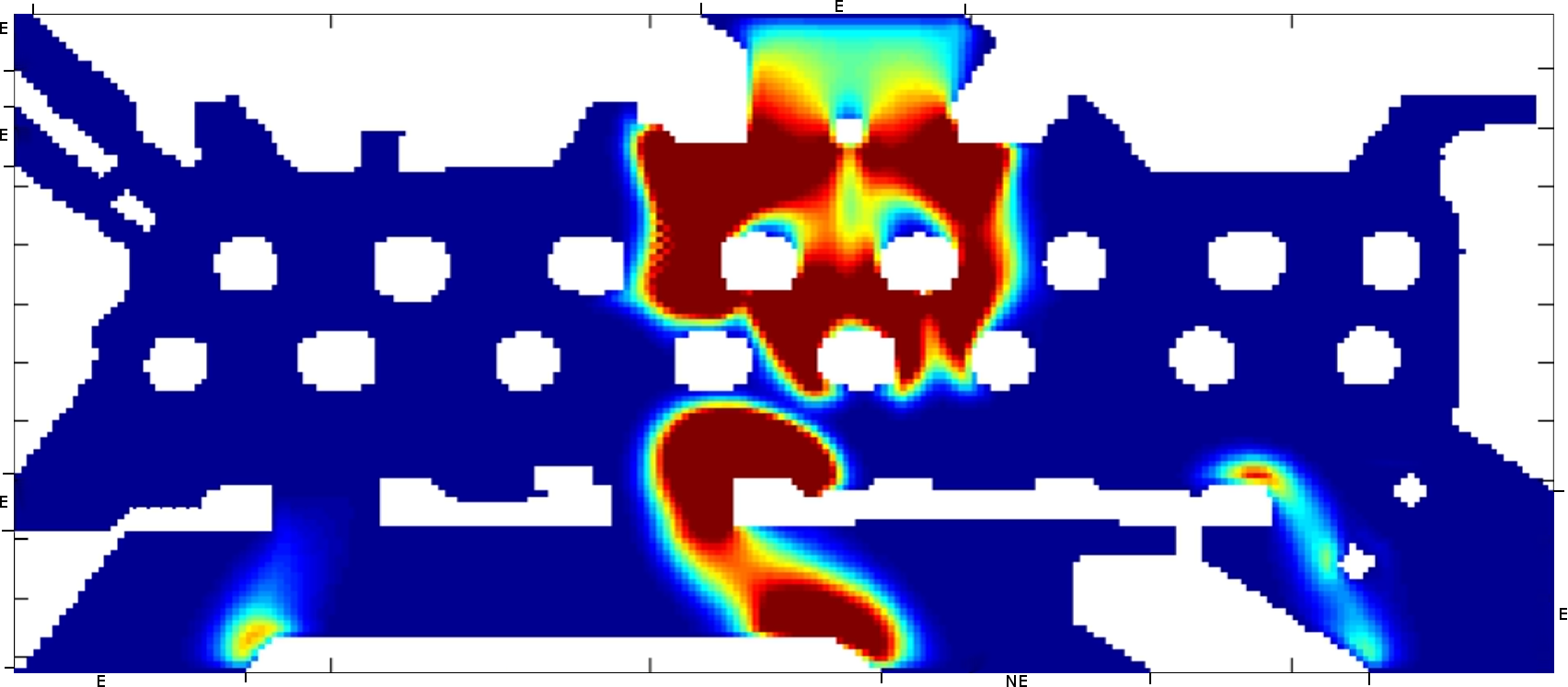}
\includegraphics[height=4.5cm,width=0.6cm]{bar.png}
\caption{Simulation for the structure in 2016. Contour lines at time $t=1s$ (top) $t=5s$ (bottom), with $M_0=0.7$. Exits are marked with an `E'; the new exit added is marked with `NE'.} \label{fig5}
\end{center}
\end{figure}
\begin{table}[th]
\begin{center}
\begin{tabular}[b]{c|c|c|c}
        $M_0$   &  2014 configuration & 2016 configuration & improvement \\
\hline \hline
\bfseries 0.9 &  9.65 & 7.22 & 25.2\% \\
\bfseries 0.8  & 8.92 & 6.86 & 19.8\%\\
\bfseries 0.7    &  8.13&  6.52 &  22.3\% \\
\bfseries 0.6 & 7.27  & 6.18 & 14.9\% \\
\bfseries 0.5 &  6.35 & 5.82 & 8.3\% \\
\bfseries 0.4 &  5.39  &  5.02  & 6.8\%\\
\end{tabular}
\end{center}
\caption{Comparison of the evacuation time for different values of the total initial mass.} \label{fig66}
\end{table}

\vspace{-2cm} 
\section*{Acknowledgements}
\noindent AF and MTW acknowledge financial support from the Austrian Academy of Sciences \"OAW via the New Frontiers Group NST-001.  {   EC and FJS benefited from the support of the ``FMJH Program Gaspard Monge in optimization and operation research'', and from the support to this program from EDF.}

\end{document}

%% file: mymacro.tex
\def\dd{{\rm d}}

\newcommand{\ov}[1]{\overline{#1}}

\def\weight(#1,#2){c_{#1,#2}}

 \if {
  
 } \fi

\if {

}{{\caption}}
} \fi

\def\at{\tilde{a}}

\def\tt{\tilde{t}}

\def\A{\mathcal{A}}
\def\B{\mathcal{B}}

\def\F{\mathcal{F}}
\def\G{\mathcal{G}}

\def\L{\mathcal{L}}

\def\T{\mathcal{T}}

\def\eps{\varepsilon}
\def\Om{{\Omega}}

\def\half{\mbox{$\frac{1}{2}$}}

\def\1B{{\bf  1}}

\newcommand{\NN}{\mathbb{N}}
\newcommand{\ZZ}{\mathbb{Z}}

\newcommand{\RR}{\mathbb{R}}

\def\EE{\mathbb{E}}
\def\PP{\mathbb{P}}

\newcommand\be{\begin{equation}}
\newcommand\ee{\end{equation}}
\newcommand\ba{\begin{array}}
\newcommand\ea{\end{array}}
\newcommand{\bean}{\begin{eqnarray*}}
\newcommand{\eean}{\end{eqnarray*}}






%% file: carlinifestasilwolf_preprint.bbl
\begin{thebibliography}{99}

\bibitem{KAF13}
Alla A, Falcone M, Kalise D (2015)
\newblock An efficient policy iteration algorithm for dynamic programming equations. 
\newblock SIAM J Sci Comput 37(1):A181--A200

\bibitem{amadori2012one}
Amadori D, Di~Francesco M (2012)
\newblock The one-dimensional {H}ughes model for pedestrian flow:
  {R}iemann-type solutions.
\newblock  Acta Math Sci 32(1):259--280

\bibitem{amadori2014existence}
Amadori D, Goatin P, Rosini M D (2014)
\newblock Existence results for {H}ughes' model for pedestrian flows.
\newblock J Math Anal Appl
  420(1):387--406

\bibitem{aubin2009viability}
Aubin J P (2009)
\newblock Viability theory.
\newblock Springer Science \& Business Media

\bibitem{axelssonbook}
Axelsson O (1996)
\newblock Iterative solution methods.
\newblock Cambridge University Press, New York

\bibitem{bellomo2011modeling}
Bellomo N, Dogbe C (2011)
\newblock On the modeling of traffic and crowds: A survey of models,
  speculations, and perspectives.
\newblock SIAM Rev 53(3):409--463

\bibitem{MR3134900}
Bensoussan A, Frehse J, Yam P (2013)
\newblock Mean field games and mean field type control theory.
\newblock Springer Briefs in Mathematics, Springer, New York

\bibitem{Ber95}
Bertsekas D (1995)
\newblock Dynamic programming and optimal control.
\newblock Athena Scientific, Belmont, Massachusetts


\bibitem{blue2001cellular}
Blue V J, Adler J L (2001)
\newblock Cellular automata microsimulation for modeling bi-directional
  pedestrian walkways.
\newblock Trasport Res B-Meth 35(3):293--312

\bibitem{BG2014}
Blandin S, Goatin P (2016)
\newblock Well-posedness of a conservation law with non-local flux arising in traffic flow modeling.
\newblock Numer Math 132(2):217--241

\bibitem{Bossy2004}
Bossy M, Gobet E, Talay D (2004)
\newblock A symmetrized {E}uler scheme for an efficient approximation of
  reflected diffusions.
\newblock J Appl Probab 41(3):877--889

\bibitem{MR3199781}
Burger M, Di~Francesco M, Markowich P A,   Wolfram M T (2014)
\newblock Mean field games with nonlinear mobilities in pedestrian dynamics.
\newblock Discrete Contin Dyn Syst Ser B 19(5):1311--1333

\bibitem{burstedde2001simulation}
Burstedde C, Klauck K, Schadschneider A, Zittartz J (2001)
\newblock Simulation of pedestrian dynamics using a two-dimensional cellular
  automaton.
\newblock Physica A 295(3):507--525

\bibitem{CacFal15}
Cacace S, Falcone M (2016)
\newblock A dynamic domain decomposition for the eikonal-diffusion equation. 
\newblock  Discret Contin Dyn S 9(1):109--123

\bibitem{CamFal95}
Camilli F, Falcone M (1995)
\newblock An approximation scheme for the optimal control of diffusion
  processes.
\newblock RAIRO Mod\'el Math Anal Num\'er 29(1):97--122

\bibitem{CS13}
Carlini E,  Silva F J (2013)
\newblock Semi-{L}agrangian schemes for mean field game models.
\newblock Decision and Control (CDC) 2013 IEEE 52nd Annual Conference 3115--3120

\bibitem{CS12}
Carlini E, Silva F J (2014)
\newblock A fully discrete {S}emi-{L}agrangian scheme for a first order mean field
  game problem.
\newblock SIAM J Num Anal 52(1):45--67

\bibitem{CS15}
Carlini E, Silva F J (2015)
\newblock A {S}emi-{L}agrangian scheme for a degenerate second order mean field
  game system.
\newblock  Discret Contin Dyn S  35(9):4269--4292

\bibitem{MR3395471}
Carmona R, Delarue F (2015)
\newblock Forward--backward stochastic differential equations and controlled
  {M}c{K}ean-{V}lasov dynamics.
\newblock Ann Probab 43(5):2647--2700

\bibitem{colombo2005pedestrians}
Colombo R M, Rosini M D (2005)
\newblock Pedestrian flows and non-classical shocks.
\newblock Math Method Appl Sci 
  28(13):1553--1567


\bibitem{colombo2012}
Colombo R M, L{\'e}cureux-Mercier M (2012)
\newblock Nonlocal crowd dynamics models for several populations.
\newblock Acta Mathematica Scientia 32(1):177--196


\bibitem{cristiani2007fast}
Cristiani E, Falcone M (2007)
\newblock Fast {S}emi-{L}agrangian schemes for the eikonal equation and
  applications.
\newblock SIAM J Num Anal 45(5):1979--2011

\bibitem{CPT_BOOK}
Cristiani E, Piccoli B, Tosin A (2014)
\newblock  Multiscale modeling of pedestrian dynamics.
\newblock MS\&A: Modeling, Simulations and Applications, Vol. 12, Springer

\bibitem{CPT15}
Cristiani E, Priuli F S, Tosin A (2015)
\newblock Modeling rationality to control self-organization of crowds: an
  environmental approach.
\newblock SIAM J Appl Math, 75(2):605--629


\bibitem{Degond}
Degond P, Appert-Rolland C, Pettr{\'e} J, Theraulaz G (2013)
 \newblock Vision-based macroscopic pedestrian models.
 \newblock Kinet Relat Models 6(4):809--839

\bibitem{di2011hughes}
Di~Francesco M, Markowich P A, Pietschmann F P, Wolfram M T (2011)
\newblock On the {H}ughes' model for pedestrian flow: {T}he one-dimensional
  case.
\newblock J Differ Equations 250(3):1334--1362

\bibitem{falconeferrettilibro}
Falcone M, Ferretti R (2013)
\newblock Semi-{L}agrangian {A}pproximation {S}chemes for {L}inear and
{H}amilton-{J}acobi {E}quations.
\newblock MOS-SIAM Series on Optimization
 
 
\bibitem{FleSon92}
Fleming W H, Soner H M (1993)
\newblock  Controlled Markov processes and viscosity solutions.
\newblock Springer, New York
 
\bibitem{Gobet2001}
Gobet E (2000)
\newblock Weak approximation of killed diffusion using {E}uler schemes.
\newblock Stochastic Process Appl 87(2):167--197

\bibitem{helbing1995social}
Helbing D, Molnar P (1995)
\newblock Social force model for pedestrian dynamics.
\newblock Phys Rev E 51:4282

\bibitem{huang2006large}
Huang M, Malham{\'e} R P, Caines P E (2006)
\newblock Large population stochastic dynamic games: closed-loop {M}c{K}ean-{V}lasov
  systems and the {N}ash certainty equivalence principle.
\newblock Commun Inf Syst 6(3):221--252

\bibitem{hughes2002continuum}
Hughes L R (2002)
\newblock A continuum theory for the flow of pedestrians.
\newblock Trasport Res B-Meth 36(6):507--535

\bibitem{MR1653393}
Jourdain B, M{\'e}l{\'e}ard S (1998)
\newblock Propagation of chaos and fluctuations for a moderate model with
  smooth initial data.
\newblock Ann Inst H Poincar\'e Probab Statist 34(6):727--766


\bibitem{Lachapelle10}
Lachapelle A, Wolfram M T (2011)
\newblock On a mean field game approach modeling congestion and aversion in
  pedestrian crowds.
\newblock Transport Res B-Meth 45(10):1572--1589

\bibitem{LasryLions07}
Lasry J M, Lions P L (2007)
\newblock Mean field games.
\newblock Jpn J Math 2:229--260

\bibitem{liu2015modeling}
Liu Y, Sun C,  Bie Y (2015)
\newblock Modeling unidirectional pedestrian movement: An investigation of
  diffusion behavior in the built environment.
\newblock Math Prob Eng  308261

\bibitem{MR0221595}
McKean H P (1966)
\newblock A class of {M}arkov processes associated with nonlinear parabolic
  equations.
\newblock Proc Nat Acad Sci USA 56:1907--1911

\bibitem{MR0233437}
McKean H P (1967)
\newblock Propagation of chaos for a class of non-linear parabolic equations.
\newblock Air Force Office Sci Res Arlington Va  41--57

\bibitem{MR1431299}
M{\'e}l{\'e}ard S (1996)
\newblock Asymptotic behaviour of some interacting particle systems;
  {M}c{K}ean-{V}lasov and {B}oltzmann models.
\newblock  Lecture Notes in Math, Springer, Berlin, 42--95

\bibitem{falcone06}
Falcone M (2006)
\newblock Numerical methods for differential games based on partial
  differential equations.
\newblock Int Game Theory Rev 8(2):231--272

\bibitem{piccoli2011time}
Piccoli B, Tosin A (2011)
\newblock Time-evolving measures and macroscopic modeling of pedestrian flow.
\newblock Arch Ration Mech An 199(3):707--738 

\bibitem{Protter-book}
Protter P (2005)
\newblock  Stochastic Integration and Differential Equations.
\newblock Springer-Verlag, Heidelberg

\bibitem{puterman1979convergence}
Puterman M L, Brumelle S L (1979)
\newblock On the convergence of policy iteration in stationary dynamic
  programming.
\newblock Math  Oper Res, 4(1):60--69


\bibitem{santos2004convergence}
Santos M S, Rust J (2004)
\newblock Convergence properties of policy iteration.
\newblock SIAM J Control Optim 42(6):2094--2115

\bibitem{Set99}
Sethian J A (1999)
\newblock  Level sets methods and fast marching methods.
\newblock Cambridge University Press, Cambridge

\bibitem{MR1108185}
Sznitman A S (1991)
\newblock Topics in propagation of chaos.
\newblock  Lecture Notes in Math, Springer, Berlin,  165--251

\bibitem{thompson1995computer}
Thompson P A, Marchant E W (1995)
\newblock Computer and fluid modelling of evacuation.
\newblock  Safety Sci 18(4):277--289

\bibitem{tsitsiklis1995}
Tsitsiklis J N (1995)
\newblock Efficient algorithms for globally optimal trajectories.
\newblock IEEE T Automat Contr 40(9):1528--1538

\bibitem{Zhao05}
Zhao H (2004)
\newblock A fast sweeping method for eikonal equations.
\newblock  Math Comput 74:603--627

\end{thebibliography}
